\documentclass[11pt,a4paper]{article}

\usepackage{tikz}
\usepackage{amsthm}
\usepackage{amsmath}
\usepackage{amssymb}
\usepackage{mathrsfs}
\usepackage{geometry}
\usepackage{graphicx}
\usepackage{amsfonts}
\usepackage{epstopdf}
\usepackage{enumerate}
\usepackage{hyperref}
\usepackage{subcaption}
\usepackage{color,soul}

\usetikzlibrary{decorations.pathreplacing}

\allowdisplaybreaks

\DeclareMathOperator*{\esssup}{ess\,sup}

\newcommand{\vphi}{\varphi}
\newcommand{\veps}{\varepsilon}

\newcommand{\dd}{\partial}
\newcommand{\md}{\mathrm{d}}

\newcommand{\R}{\mathbb{R}}

\newcommand{\Rmnum}[1]{\uppercase\expandafter{\romannumeral#1}} 

\newcommand{\bbE}{\mathbb{E}}
\newcommand{\bbP}{\mathbb{P}}
\newcommand{\bbZ}{\mathbb{Z}}

\newcommand{\calE}{\mathcal{E}}
\newcommand{\calF}{\mathcal{F}}

\newcommand{\calK}{\mathcal{K}}

\newcommand{\calN}{\mathcal{N}}

\newcommand{\myset}[1]{\left\{#1\right\}}
\newcommand{\mybar}[1]{\overline{#1}}
\newcommand{\mynorm}[1]{\lVert#1\rVert}
\newcommand{\mytilde}[1]{\widetilde{#1}}

\newcommand{\mysupp}{\mathrm{supp}}

\def\Xint#1{\mathchoice
{\XXint\displaystyle\textstyle{#1}}%
{\XXint\textstyle\scriptstyle{#1}}%
{\XXint\scriptstyle\scriptscriptstyle{#1}}%
{\XXint\scriptscriptstyle\scriptscriptstyle{#1}}%
\!\int}
\def\XXint#1#2#3{{\setbox0=\hbox{$#1{#2#3}{\int}$ }
\vcenter{\hbox{$#2#3$ }}\kern-.6\wd0}}

\def\dashint{\Xint-}

\newtheorem{mythm}{Theorem}[section]
\newtheorem{myprop}[mythm]{Proposition}
\newtheorem{mylem}[mythm]{Lemma}
\newtheorem{mycor}[mythm]{Corollary}
\newtheorem{myrmk}[mythm]{Remark}

\usepackage{fancybox}

\AtEndDocument{
\bigskip{\footnotesize%
  \textsc{Universit\'e Grenoble Alpes, CNRS UMR 5582, Institut Fourier, Gi\`eres, France.} \par  
  \textit{E-mail address}: \texttt{baptiste.devyver@univ-grenoble-alpes.fr} \par
}
\bigskip{\footnotesize%
  \textsc{Universit\'e Grenoble Alpes, CNRS UMR 5582, Institut Fourier, Gi\`eres, France.} \par  
  \textit{E-mail address}: \texttt{emmanuel.russ@univ-grenoble-alpes.fr} \par
}
\bigskip{\footnotesize%
  \textsc{Centro de Matem\'atica e Aplica\c{c}\~oes, Faculdade de Ci\^encias e Tecnologia, Universidade Nova de Lisboa, Quinta da Torre, 2829--516 Caparica, Portugal} \par
  \textit{E-mail address}: \texttt{yangmengqh@gmail.com}
}}

\begin{document}

\title{{\Large{\bf{Gradient Estimate for the Heat Kernel on Some Fractal-Like Cable Systems and Quasi-Riesz Transforms}}}}
\author{Baptiste Devyver, Emmanuel Russ and Meng Yang}
\date{}

\maketitle

\abstract{We give pointwise upper estimate for the gradient of the heat kernel on some fractal-like cable systems including the Vicsek and the Sierpi\'nski cable systems. Applications to $L^p$-boundedness of quasi-Riesz transforms are derived.}

\renewcommand{\thefootnote}{}

\footnote{\textsl{Date}: \today}

\renewcommand{\thefootnote}{\alph{footnote}}
\setcounter{footnote}{0}

\section{Introduction}\label{sec_intro}

Throughout the paper, the letters $C,C_1,C_2,C_A,C_B$ will always refer to some positive constants and may change at each occurrence. The sign $\asymp$ means that the ratio of the two sides is bounded from above and below by positive constants. The sign $\lesssim$ ($\gtrsim$) means that the LHS is bounded by positive constant times the RHS from above (below).

On a complete non-compact Riemannian manifold, a celebrated result independently discovered by Grigor'yan \cite{Gri92} and Saloff-Coste \cite{Sal92,Sal95} is that the following two-sided Gaussian bound of the heat kernel
\begin{equation}\label{eqn_Gauss}
\frac{C_1}{V(x,\sqrt{t})}\exp\left(-C_2\frac{d(x,y)^2}{t}\right)\le p_t(x,y)\le\frac{C_3}{V(x,\sqrt{t})}\exp\left(-C_4\frac{d(x,y)^2}{t}\right),
\end{equation}
where $V(x,r)$ denotes the Riemannian measure of the open geodesic ball with center $x$ and radius $r$, is equivalent to the conjunction of the volume doubling condition and the scale-invariant $L^2$-Poincar\'e inequality on balls. The proofs go through a parabolic Harnack inequality for the solutions of the heat equation, which is itself equivalent to (\ref{eqn_Gauss}).

However, the matching upper estimate of the gradient of the heat kernel
$$|\nabla_yp_t(x,y)|\le\frac{C_1}{\sqrt{t}V(x,\sqrt{t})}\exp\left(-C_2\frac{d(x,y)^2}{t}\right)$$
only holds in some cases, for example, Riemannian manifolds with non-negative Ricci curvature \cite{LY86}, Lie groups with polynomial volume growth \cite{Sal90} and covering manifolds with polynomial volume growth \cite{Dun04a,Dun04b}.

Pointwise or (weighted) $L^p$-bounds for the gradient of the heat kernel play an important role in the $L^p$-boundedness of the Riesz transform for $p>2$. On a complete non-compact Riemannian manifold, it is obvious that $\mynorm{|\nabla u|}_2=\mynorm{\Delta^{1/2}u}_2$ for any smooth function $u$ with compact support, hence the Riesz transform $\nabla\Delta^{-1/2}$ is $L^2$-bounded. Strichartz \cite{Str83} formulated the following question: for which values of $p\in (1,+\infty)$ is the Riesz transform $\nabla\Delta^{-1/2}$ $L^p$-bounded? A celebrated result was given by Coulhon and Duong \cite[Theorem 1.1]{CD99} that the volume doubling condition and a pointwise on-diagonal upper bound of the heat kernel imply the $L^p$-boundedness of the Riesz transform for any $p\in(1,2]$. Moreover, this conclusion is false when $p>2$, as the counterexample given by the connected sum of two copies of $\R^n$ shows (\cite[Section 5]{CD99}). For $p>2$, Auscher, Coulhon, Duong and Hofmann \cite[Theorem 1.3]{ACDH04} proved that, under the volume doubling condition and the two-sided Gaussian bound of the heat kernel, the $L^p$-estimate of the gradient of the heat kernel is equivalent to the $L^p$-boundedness of the Riesz transform in some proper sense. Recently, Coulhon, Jiang, Koskela and Sikora \cite{CJKS20} generalized the above result to metric measure spaces endowed with a Dirichlet form deriving from a ``carr\'e du champ" and improved the results of \cite{ACDH04}, even in the case of Riemannian manifolds.

Fractals provide new examples with very different phenomena. One important estimate is the so-called sub-Gaussian bound as follows.
$$\frac{C_1}{V(x,t^{1/\beta})}\exp\left(-C_2\left(\frac{d(x,y)}{t^{1/\beta}}\right)^{\frac{\beta}{\beta-1}}\right)\le p_t(x,y)\le\frac{C_3}{V(x,t^{1/\beta})}\exp\left(-C_4\left(\frac{d(x,y)}{t^{1/\beta}}\right)^{\frac{\beta}{\beta-1}}\right),$$
where $\beta$ is a new parameter called the walk dimension which is always strictly greater than 2 on fractals. For example, on the Sierpi\'nski gasket (see Figure \ref{fig_SG}), $\beta=\log5/\log2$, see \cite{BP88,Kig89}, on the Sierpi\'nski carpet (see Figure \ref{fig_SC}), $\beta\approx2.09697$, see \cite{BB89,BB90,BBS90,BB92,HKKZ00,KZ92}.

\begin{figure}[ht]
\centering
\begin{minipage}[t]{0.48\textwidth}
\centering
\includegraphics[width=\textwidth]{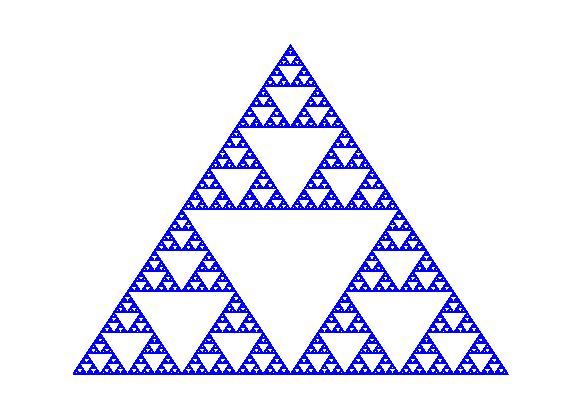}
\caption{The Sierpi\'nski Gasket}\label{fig_SG}
\end{minipage}
\begin{minipage}[t]{0.48\textwidth}
\centering
\includegraphics[width=\textwidth]{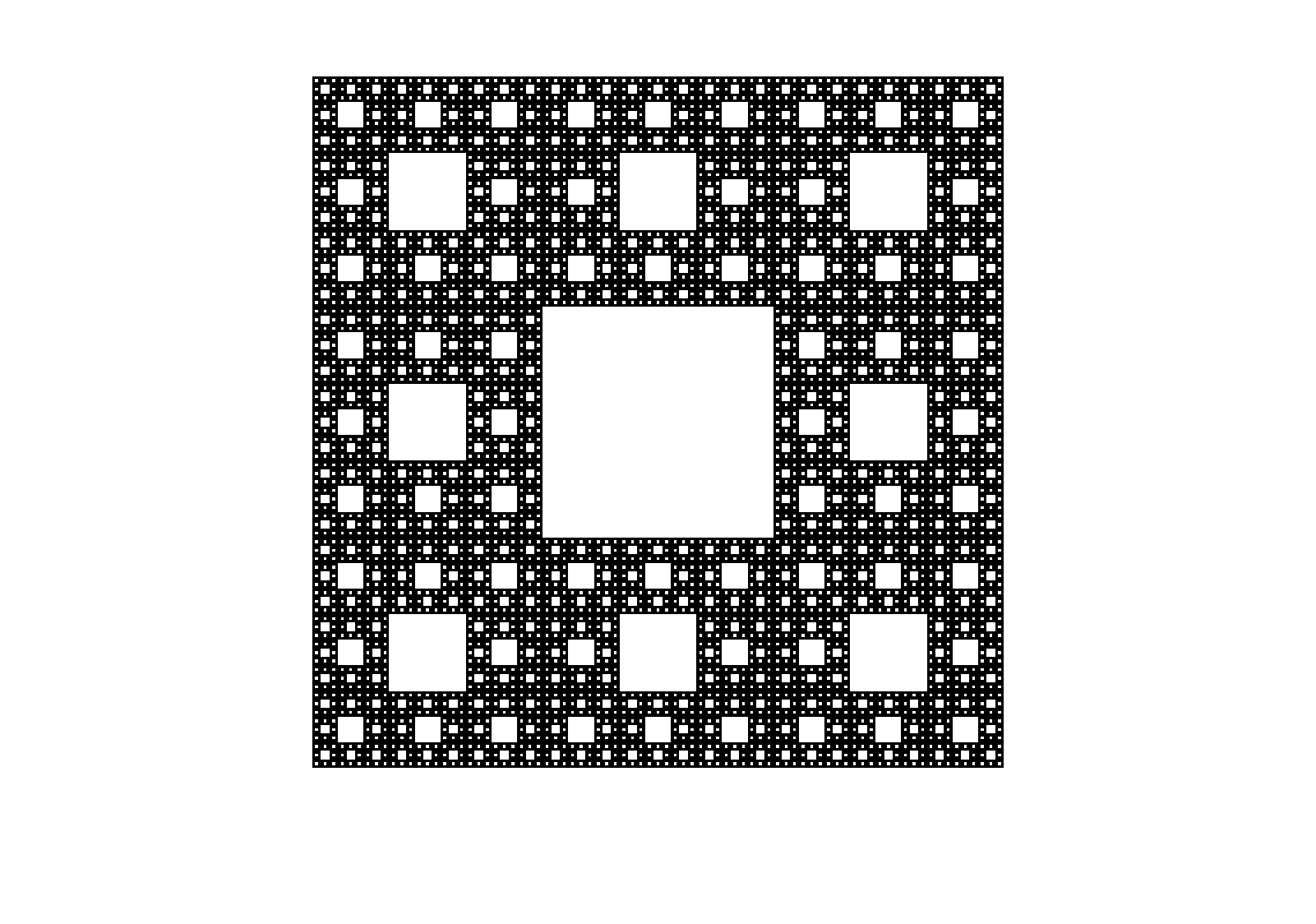}
\caption{The Sierpi\'nski Carpet}\label{fig_SC}
\end{minipage}
\end{figure}

A natural question is to consider the matching upper estimate of the gradient of the heat kernel. However, gradient operator can not be easily defined using the classical Euclidean way due to the existence of too many ``holes". We turn to consider the corresponding fractal-like manifolds or fractal-like cable systems. Roughly speaking, given a fractal, by translating the small scale self-similar property, we obtain an infinite graph with self-similar property in the large scale. If we replace each edge of the graph by a tube and glue these tubes smoothly at each vertex, then we obtain a fractal-like manifold where the gradient operator is the standard one on a Riemannian manifold. If we replace each edge of the graph by an interval, then we obtain a fractal-like cable system where the gradient operator can be defined as the usual derivative on each interval (although only one-sided derivatives are well-defined at the endpoints of each interval, it does not matter since the set of all the endpoints has measure zero in our consideration).

On a fractal-like manifold or a fractal-like cable system, one can consider the Riesz transform. Chen, Coulhon, Feneuil and the second author \cite{CCFR17}\footnote{Actually, only the cases of manifolds and graphs were considered in \cite{CCFR17}, but the same methods could easily yield the corresponding results in the case of cable systems.} proved that the volume doubling condition and the sub-Gaussian heat kernel upper bound\footnote{More precisely, a Gaussian upper bound is assumed for small $t$, while a sub-Gaussian one is assumed for large $t$.} imply the $L^p$-boundedness of the Riesz transform for any $p\in(1,2]$. They also proved that in the Vicsek case, the Riesz transform is $L^p$-bounded if and only if $p\in(1,2]$, where the fact that the Vicsek set is a tree was intrinsically used to do some explicit calculations of the $L^p$-norms of harmonic functions. Amenta \cite{Ame21} generalized the $L^p$-unboundedness for $p>2$ to other Riemannian manifolds that satisfy the so-called spinal condition which can be regarded as a weaker form of the tree condition.

In this paper, we consider some fractal-like cable systems satisfying certain geometric and functional assumptions, and obtain pointwise upper bound for the gradient of the heat kernel as well as the $L^p$-boundedness for the ``quasi-Riesz transforms", which are a suitable variant of the Riesz transform.

Our first main result is a pointwise gradient upper estimate for the heat kernel on a cable system $X$ under some assumptions, which we only roughly describe here (the precise definitions of which will be given in Section \ref{sec_MMD}). The assumptions on $X$ involve two positive parameters $\alpha$ and $\beta$ satisfying $2\le\beta\le\alpha+1$ and two functions $\Phi$ and $\Psi$ defined on $(0,+\infty)$ given by
\begin{equation}\label{eqn_PhiPsi}
\begin{split}
\Phi(r)&=r1_{(0,1)}(r)+r^\alpha1_{[1,+\infty)}(r),\\
\Psi(r)&=r^21_{(0,1)}(r)+r^{\beta}1_{[1,+\infty)}(r).
\end{split}
\end{equation}
Denote by $d$ (resp. $m$) the metric (resp. the measure) on $X$. Consider also a strongly local  regular Dirichlet form $(\calE,\calF)$ on $L^2(X;m)$. Let $V(x,r)$ denote the measure of the open ball with center $x\in X$ and radius $r\in(0,+\infty)$. The assumption V($\Phi$) means that $V(x,r)\asymp \Phi(r)$.

To establish pointwise gradient estimate for the heat kernel, we need to assume a pointwise upper bound for the heat kernel itself, that is, say that UHK($\Psi$) holds if
$$
p_t(x,y)\le\frac{1}{V\left(x,\Psi^{-1}(C_1t)\right)}\exp\left(-\Upsilon\left(C_2d(x,y),t\right)\right),
$$
where
$$
\Upsilon(R,t)=\sup_{s\in(0,+\infty)}\left(\frac{R}{s}-\frac{t}{\Psi(s)}\right),
$$
and a reverse H\"older inequality for the gradients of harmonic functions. Recall that such inequalities play a key role in the study of the $L^p$-boundedness of the Riesz transform for $p>2$, see \cite{AC05,BF16,CJKS20}. The classical statement for the reverse H\"older inequality is as follows: say that the reverse H\"older inequality for the gradients of harmonic functions holds if there exists $C_H\in(0,+\infty)$ such that for any ball $B$ with radius $r$, for any function $u$ harmonic in $2B$, we have 
$$\left\Vert \left\vert \nabla u\right\vert\right\Vert_{L^{\infty}(B)}\le\frac{C_H}r\dashint_{2B} \left\vert u\right\vert\md m$$
(the precise definition of $\left\vert \nabla u\right\vert$ will be given at the end of Section \ref{sec_MMD}). The reverse H\"older inequality holds on the $N$-dimensional Vicsek cable system for any $N\ge 2$, see Proposition \ref{prop_RH_Vicsek}. However, this condition does not hold on the Sierpi\'nski cable system, see Proposition \ref{prop_RH}. This leads us to introduce a new condition, called a generalized reverse H\"older inequality for the gradients of harmonic functions, which also holds on the Vicsek cable systems, see Proposition \ref{prop_GRH}. More precisely, say that the generalized reverse H\"older inequality GRH($\Phi,\Psi$) holds if there exists $C_H\in(0,+\infty)$ such that for any ball $B$ with radius $r$, for any function $u$ harmonic in $2B$, we have
$$\mynorm{|\nabla u|}_{L^\infty(B)}\le C_H\frac{\Phi(r)}{\Psi(r)}\dashint_{2B}|u|\md m.$$

Our first main result states as follows.

\begin{mythm}\label{thm_grad}
Let $(X,d,m,\calE,\calF)$ be an unbounded cable system satisfying \ref{eqn_VPhi}, \ref{eqn_UHK} and \ref{eqn_GRH}. We have the gradient estimate \ref{eqn_GHK} for the heat kernel as follows: there exist $C_1,C_2,C_3,C_4\in(0,+\infty)$ such that for any $t\in(0,+\infty)$, for $m$-a.e. $x,y\in X$, we have
\begin{align}
|\nabla_yp_t(x,y)|&\le\frac{\Phi(\Psi^{-1}(t))}{tV\left(x,\Psi^{-1}(C_1t)\right)}\exp\left(-\Upsilon\left(C_2d(x,y),t\right)\right)\nonumber\\
&\le
\begin{cases}
\frac{C_3}{\sqrt{t}V(x,\sqrt{t})}\exp\left(-C_4\frac{d(x,y)^2}{t}\right),&\text{if }t\in(0,1),\\
\frac{C_3}{t^{1-\frac{\alpha}{\beta}}V(x,t^{1/\beta})}\exp\left(-C_4\left(\frac{d(x,y)}{t^{{1}/{\beta}}}\right)^{\frac{\beta}{\beta-1}}\right),&\text{if }t\in[1,+\infty).
\end{cases}
\label{eqn_GHK}\tag*{GHK($\Phi,\Psi$)}
\end{align}
In particular, \ref{eqn_GHK} holds on the $N$-dimensional Vicsek cable system with $\alpha=\log(2^N+1)/\log3$, $\beta=\log(3\cdot(2^N+1))/\log3$ for any $N\ge 2$ and on the Sierpi\'nski cable system with $\alpha=\log3/\log2$ and $\beta=\log5/\log2$.
\end{mythm}

To the best of our knowledge, this is the first pointwise upper bound for the gradient of the heat kernel in a sub-Gaussian type context.

\begin{myrmk}
{\em 
Arguing as in the proof of Theorem \ref{thm_grad}, we also prove the following result. For any $p\in(1,+\infty)$, there exist $\gamma,C\in(0,+\infty)$ such that for $m$-a.e. $y\in X$, we have
$$\left\Vert |\nabla p_t(\cdot,y)|\exp\left(\gamma\frac{d(\cdot,y)^2}{t}\right)\right\Vert_{p}\le\frac{C}{\sqrt{t}\left[V(y,\sqrt{t})\right]^{1-\frac{1}{p}}}\text{ if }t\in(0,1),$$
and
\begin{equation}\label{eq:Lpgrad1}
\left\Vert |\nabla p_t(\cdot,y)|\exp\left(\gamma\left(\frac{d(\cdot,y)}{t^{1/\beta}}\right)^{\frac{\beta}{\beta-1}}\right)\right\Vert_{p}\le\frac{C}{t^{1-\frac{\alpha}{\beta}}\left[V(y,t^{1/\beta})\right]^{1-\frac{1}{p}}}\text{ if }t\in[1,+\infty).
\end{equation}
For $p\in(1,2)$, it was proved in \cite[Lemma 2.2]{CCFR17} that, on any general Riemannian manifold or graph $M$ satisfying \ref{eqn_VD} and \ref{eqn_UHK} (no reverse H\"older inequality is required), the following weighted $L^p$-estimate for the gradient of the heat kernel holds: there exist $\gamma,C\in(0,+\infty)$ such that for any $y\in M$, we have
$$\left\Vert \left\vert\nabla p_t(\cdot,y)\right\vert \exp\left(\gamma\frac{d(\cdot,y)^2}{t}\right)\right\Vert_p\le\frac {C}{\sqrt{t}\left[V(y,\sqrt{t})\right]^{1-\frac{1}{p}}}\text{ if }t\in(0,1),$$
and
\begin{equation}\label{eq:Lpgrad2}
\left\Vert \left\vert \nabla p_t(\cdot,y)\right\vert \exp\left(\gamma\left(\frac{d(\cdot,y)}{t^{1/\beta}}\right)^{\frac{\beta}{\beta-1}}\right)\right\Vert_p\le\frac{C}{\sqrt{t}\left[V(y,t^{1/{\beta}})\right]^{1-\frac 1p}}\text{ if }t\in[1,+\infty).
\end{equation}
Since $2\alpha\ge \alpha+1\ge \beta$, which entails $1-\frac{\alpha}{\beta}\le\frac{1}{2}$, it follows that for $p\in(1,2)$, the conclusion of Theorem \ref{thm_grad} is weaker than the one of \cite[Lemma 2.2]{CCFR17}, that is, (\ref{eq:Lpgrad1}) is weaker than (\ref{eq:Lpgrad2}). On the other hand, (\ref{eq:Lpgrad2}) is limited to the range $p\in (1,2)$, and no pointwise estimate for the gradient of the heat kernel was proved in \cite{CCFR17}. In fact, it is unlikely that the corresponding pointwise estimate for the gradient of the heat kernel holds under \ref{eqn_VD} and \ref{eqn_UHK} only.
}
\end{myrmk}

From Theorem \ref{thm_grad}, we easily obtain another bound for $\left\vert\nabla p_t\right\vert$, assuming a two-sided bound for $p_t$. Namely, say that the heat kernel bound HK($\Psi$) holds if
\begin{footnotesize}
$$\frac{1}{V\left(x,\Psi^{-1}(C_1t)\right)}\exp\left(-\Upsilon\left(C_2d(x,y),t\right)\right)\le p_t(x,y)\le \frac{1}{V\left(x,\Psi^{-1}(C_3t)\right)}\exp\left(-\Upsilon\left(C_4d(x,y),t\right)\right).$$
\end{footnotesize}
Then, it is plain to deduce from Theorem \ref{thm_grad} the following result:
\begin{mycor}
Let $(X,d,m,\calE,\calF)$ be an unbounded cable system satisfying \ref{eqn_VPhi}, \hyperlink{eqn_HKPsi}{HK($\Psi$)} and \ref{eqn_GRH}. Then there exist $C_1,C_2\in(0,+\infty)$ such that for any $t\in(0,+\infty)$, for $m$-a.e. $x,y\in X$, we have
$$\left\vert \nabla_yp_t(x,y)\right\vert\le C_1\frac {\Phi(\Psi^{-1}(t))}{t}p_{C_2t}(x,y).$$
\end{mycor}

Note that this is similar to the classical estimate for $\left\vert \nabla p_t\right\vert$ on Riemannian manifolds with non-negative Ricci curvature (with $C_1=C_2=1$ in this case), see \cite{LY86}.

The strategy of our proof is to use the fact that the heat kernel is a solution of the heat equation which can be regarded as a Poisson equation for a fixed time. Since the regularity of the time derivative of the heat kernel is easy to handle, one only needs to have gradient estimates for the solutions of Poisson equation; this has been the approach considered in \cite{JKY14,CJKS20}. In their settings, the local quantitative Lipschitz regularity for Cheeger-harmonic functions \cite[Theorem 3.1]{JKY14} or the reverse H\"older inequality for the gradients of harmonic functions \cite[Theorem 3.2]{CJKS20} which are consequences of some curvature assumptions was needed. In the present work, we rely on the approach through the generalized reverse H\"older inequality, see Proposition \ref{prop_Poi_grad}.

Our second main result is the $L^p$-boundedness of the ``quasi-Riesz transforms'' for any $p\in (1,+\infty)$ in the fractal setting. Recall that, under the volume doubling condition and the sub-Gaussian pointwise upper bound for the heat kernel, it was proved in \cite[Theorem 1.2]{CCFR17} that the Riesz transform $\nabla \Delta^{-1/2}$ is $L^p$-bounded for any $p\in(1,2]$. In \cite{ACDH04}, the $L^p$-boundedness of the Riesz transform for some values of $p>2$ was derived from an $L^p$-bound for the gradient of the heat semi-group, where the volume doubling condition and the scale-invariant $L^2$-Poincar\'e inequality are assumed to hold (recall that, in this situation, pointwise Gaussian upper and lower bounds for the heat kernel hold). This contrasts with the {\it sub-Gaussian} context. Indeed, in the case of the Vicsek cable systems, the Riesz transform is $L^p$-unbounded for any $p>2$, see \cite[Theorem 5.1]{CCFR17}, while a pointwise sub-Gaussian bound for the gradient of the heat kernel holds, as Theorem \ref{thm_grad} shows (note that an $L^2$-Poincar\'e inequality on balls, with a nonstandard scaling, still holds in this case).

Motivated by situations where the heat kernel satisfies sub-Gaussian upper bound, consider a weaker form of the Riesz transform, namely the {\it quasi-Riesz transform}, where the exponent ${1}/{2}$ is replaced by some $\varepsilon\in(0,{1}/{2})$. The quasi-Riesz transforms were introduced in \cite{Chen15}, as a substitute of the Riesz transform. We also rely on the decomposition of the Riesz transform as the sum of a local part and a global one (this idea goes back to \cite{A}), which was also pushed further in \cite{Chen15}. More precisely, if $p\in (1,+\infty)$, say that the local Riesz transform is $L^p$-bounded if and only if
$$\lVert |\nabla f|\rVert_p\lesssim \lVert (I+\Delta)^{{1}/{2}}f\rVert_p$$
and that the Riesz transform at infinity is $L^p$-bounded if and only if
$$\lVert|\nabla e^{-\Delta}f|\rVert_p\lesssim \lVert \Delta^{{1}/{2}}f\rVert_p.
$$
Then the Riesz transform $\nabla \Delta^{-1/2}$ is $L^p$-bounded if and only if the local Riesz transform and the Riesz transform at infinity are both $L^p$-bounded, see \cite[Theorem 2.3]{Chen15}. Moreover, the quasi-Riesz transform at infinity, namely $\nabla e^{-\Delta} \Delta^{-\veps}$, with $\veps\in(0,1/2)$, was proved in \cite[Theorem 1.2]{Chen15} to be $L^p$-bounded for any $p\in (1,2]$ on any complete Riemannian manifold, without any further assumption.

Inspired by the aforementioned equivalence for the $L^p$-boundedness of the Riesz transform and following the approach of \cite[Section 3]{ACDH04}, we establish the $L^p$-boundedness of the quasi-Riesz transforms for any $p\in (1,+\infty)$, as a consequence of the above gradient estimate for the heat kernel.

\begin{mythm}\label{thm_Riesz}
Let $(X,d,m,\calE,\calF)$ be an unbounded cable system satisfying the same assumptions as in Theorem \ref{thm_grad}. Then for any $\veps\in(0,1-\frac{\alpha}{\beta})$, the quasi-Riesz transform $\nabla(I+\Delta)^{-1/2}+\nabla e^{-\Delta}\Delta^{-\veps}$ is $L^p$-bounded for any $p\in(1,+\infty)$.
\end{mythm}

An interesting open question is to know whether the conclusion of Theorem \ref{thm_Riesz} also holds for $\varepsilon=1-\frac{\alpha}{\beta}$. The exponent $\frac{\alpha}{\beta}$ in the gradient estimate for $t>1$ given by Theorem \ref{thm_grad} suggests that this is the case, but the result seems more difficult to prove.

We will see from the proof of Proposition \ref{prop_GRH} that these results extend to the cable systems corresponding to the p.c.f. self-similar sets considered in \cite{THP14} without any technical difficulty. However, it seems that more advanced techniques would be required to do the generalization to the cable system corresponding to the Sierpi\'nski carpet which is a non-p.c.f. self-similar set and an infinitely ramified fractal.

This paper is organized as follows. In Section \ref{sec_MMD}, we give some results about Poisson equation on metric measure Dirichlet spaces. In Section \ref{sec_cable}, we give formal constructions of cable systems including the Vicsek and the Sierpi\'nski ones. In Section \ref{sec_GRH}, we show that the reverse H\"older inequality holds on the Vicsek cable systems but does not hold on the Sierpi\'nski cable system, and we show that a generalized reverse H\"older inequality holds on the Vicsek and the Sierpi\'nski cable systems. In Section \ref{sec_grad}, we use the generalized reverse H\"older inequality to obtain gradient estimates for the solutions of Poisson equation using which we obtain gradient estimate for the heat kernel that is Theorem \ref{thm_grad}. In Section \ref{sec_Riesz}, we prove Theorem \ref{thm_Riesz}.

\noindent {\bf Acknowledgements:} This work was partly supported by the French ANR project RAGE ANR-18-CE40-0012. The third author was supported by national funds through the FCT – Funda\c{c}\~ao para a Ci\^encia e a Tecnologia, I.P. (Portuguese Foundation for Science and Technology) within the scope of the project UIDB/00297/2020 (Centro de Matem\'atica e Aplicações).

\section{Poisson Equation on Metric Measure Dirichlet Spaces}\label{sec_MMD}

Let $(X,d,m,\calE,\calF)$ be an unbounded metric measure Dirichlet (MMD) space, that is, $(X,d)$ is a locally compact separable unbounded metric space, $m$ is a positive Radon measure on $X$ with full support, $(\calE,\calF)$ is a strongly local regular Dirichlet form on $L^2(X;m)$. Throughout this paper, we always assume that all metric balls are precompact.

For any $x\in X$, for any $r\in(0,+\infty)$, denote the (metric) ball $B(x,r)=\{y\in X:d(x,y)<r\}$, denote $V(x,r)=m(B(x,r))$. If $B=B(x,r)$, then we denote $\delta B=B(x,\delta r)$ for any $\delta\in(0,+\infty)$. Let $C(X)$ denote the space of all real-valued continuous functions on $X$ and let $C_c(X)$ denote the space of all real-valued continuous functions on $X$ with compact support.

Consider the strongly local regular Dirichlet form $(\calE,\calF)$ on $L^2(X;m)$. Let $\Delta$ be the corresponding generator which is a non-negative definite self-adjoint operator. Let $\Gamma$ be the corresponding energy measure. Denote $\calE_1(\cdot,\cdot)=\calE(\cdot,\cdot)+(\cdot,\cdot)$, where $(\cdot,\cdot)$ is the inner product in $L^2(X;m)$. We refer to \cite{FOT11} for related results about Dirichlet forms.

Let us now present some geometric and functional conditions on $(X,d,m,\calE,\calF)$ which will be used in the sequel. Most of them will be expressed in terms of the two functions $\Phi$ and $\Psi$ given by Equation (\ref{eqn_PhiPsi}). 

We say that the volume doubling condition \ref{eqn_VD} holds if there exists $C_D\in(0,+\infty)$ such that
\begin{equation*}\label{eqn_VD}\tag*{VD}
V(x,2r)\le C_DV(x,r)\text{ for any }x\in X,\text{ for any }r\in(0,+\infty).
\end{equation*}

We say that the regular volume condition \ref{eqn_VPhi} holds if there exists $C_R\in(0,+\infty)$ such that
\begin{equation*}\label{eqn_VPhi}\tag*{V($\Phi$)}
\frac{1}{C_R}\Phi(r)\le V(x,r)\le C_R\Phi(r)\text{ for any }x\in X,\text{ for any }r\in(0,+\infty).
\end{equation*}
It is obvious that \ref{eqn_VPhi} implies \ref{eqn_VD}.

\subsection{Faber-Krahn, Sobolev and Poincar\'e Inequalities}

Let $D$ be an open subset of $X$. Denote by $\lambda_1(D)$ the smallest Dirichlet eigenvalue for $D$, that is,
$$\lambda_1(D)=\inf\myset{\frac{\calE(u,u)}{\mynorm{u}_2^2}:u\in\calF_D\backslash\myset{0}},$$
where
\begin{align*}
\calF_D&=\text{ the }\calE_1\text{-closure of }\calF\cap C_c(D).
\end{align*}

We say that the relative Faber-Krahn inequality \ref{eqn_FK} holds if there exist $C_F\in(0,+\infty)$ and $\nu\in(0,1)$ such that for any ball $B=B(x,r)$, for any open subset $D$ of $B$, we have
\begin{equation*}\label{eqn_FK}\tag*{FK($\Psi$)}
\lambda_1(D)\ge\frac{C_F}{\Psi(r)}\left(\frac{m(B)}{m(D)}\right)^\nu.
\end{equation*}

We say that the local Sobolev inequality \ref{eqn_LS} holds if there exist $C_L\in(0,+\infty)$ and $q\in(2,+\infty)$ such that for any ball $B=B(x,r)$, for any $u\in\calF_{B}$, we have
\begin{equation*}\label{eqn_LS}\tag*{LS($\Psi$)}
\left(\dashint_{B}|u|^q\md m\right)^{1/q}\le C_L\sqrt{\Psi(r)}\left(\frac{1}{m(B)}\calE(u,u)\right)^{1/2},
\end{equation*}
where $\dashint_B=\frac{1}{m(B)}\int_B$.

\begin{myrmk}
{\em 
We will sometimes use the notations \hypertarget{eqn_FKPsinu}{FK($\Psi,\nu$)} and \hypertarget{eqn_LSPsiq}{LS($\Psi,q$)} to emphasize the role of the values of $\nu$ and $q$, respectively.
}
\end{myrmk}

Actually, \ref{eqn_FK} and \ref{eqn_LS} are equivalent. More precisely:

\begin{mylem}\label{lem_FKLS}
Let $(X,d,m,\calE,\calF)$ be an unbounded MMD space. Then \ref{eqn_FK} is equivalent to \ref{eqn_LS} with $q=\frac{2}{1-\nu}$ or $\nu=1-\frac{2}{q}$.
\end{mylem}

The fact that Faber-Krahn and Sobolev inequalities are equivalent in quite general contexts was already proved several times (see \cite{C96,BCLS}, for instance). Since the Faber-Krahn and the local Sobolev inequalities under consideration in the present paper involve the function $\Psi$, we provide a proof here for the sake of completeness.

\begin{proof}
The proof is inspired by \cite[Exercise 14.6]{Gri09}. ``$\Leftarrow$": Let $D$ be an open subset of a ball $B=B(x,r)$. For any $u\in\calF_D\backslash\myset{0}$, by \hyperlink{eqn_LSPsiq}{LS($\Psi,q$)}, we have
\begin{align*}
\mynorm{u}_2^2&=\int_D|u|^2\md m\le\left(\int_D|u|^{2\cdot\frac{q}{2}}\md m\right)^{\frac{2}{q}}\left(\int_D1\md m\right)^{1-\frac{2}{q}}=\left(\int_B|u|^{q}\md m\right)^{\frac{2}{q}}m(D)^{1-\frac{2}{q}}\\
&\le C_L^2\Psi(r)m(B)^{\frac{2}{q}-1}\calE(u,u)m(D)^{1-\frac{2}{q}}=C_L^2\Psi(r)\left(\frac{m(D)}{m(B)}\right)^{1-\frac{2}{q}}\calE(u,u),
\end{align*}
hence
$$\frac{\calE(u,u)}{\mynorm{u}_2^2}\ge\frac{1}{C_L^2\Psi(r)}\left(\frac{m(B)}{m(D)}\right)^{1-\frac{2}{q}}.$$
Taking the infimum with respect to $u\in\calF_D\backslash\myset{0}$, we have
$$\lambda_1(D)\ge\frac{1}{C_L^2\Psi(r)}\left(\frac{m(B)}{m(D)}\right)^{1-\frac{2}{q}},$$
that is, we have \hyperlink{eqn_FKPsinu}{FK($\Psi,\nu$)} with $\nu=1-\frac{2}{q}$ and $C_F=\frac{1}{C_L^2}$.

``$\Rightarrow$": Firstly, let $u\in\calF\cap C_c(B)$ be non-negative. Take $p\in(1,+\infty)$. For any $k\in\bbZ$, let $\Omega_k=\myset{u>2^k}$ and $m_k=m(\mybar{\Omega_k})$, then
$$\int_Bu^{2p}\md m=\sum_{k\in\bbZ}\int_{\Omega_k\backslash\Omega_{k+1}}u^{2p}\md m\le4^p\sum_{k\in\bbZ}4^{pk}m_k.$$

For any $k\in\bbZ$, let $u_k=\left(u-2^k\right)^+\wedge2^k$, then $u_k\in\calF\cap C_c(B)$ satisfies
$$\Gamma(u_k,u_k)|_{\Omega_k\backslash\Omega_{k+1}}=\Gamma(u,u)|_{\Omega_k\backslash\Omega_{k+1}},\Gamma(u_k,u_k)|_{X\backslash(\Omega_k\backslash\Omega_{k+1})}=0,$$
then
\begin{align*}
\calE(u,u)&=\sum_{k\in\bbZ}\int_{\Omega_k\backslash\Omega_{k+1}}\md\Gamma(u,u)=\sum_{k\in\bbZ}\int_{\Omega_k\backslash\Omega_{k+1}}\md\Gamma(u_k,u_k)\\
&=\sum_{k\in\bbZ}\int_{X}\md\Gamma(u_k,u_k)=\sum_{k\in\bbZ}\calE(u_k,u_k).
\end{align*}

For any $k\in\bbZ$, using the facts that $u\in C_c(B)$ and that $m$ is regular, choose an open set $\mytilde{\Omega}_k$ satisfying $\mybar{\Omega_k}\subseteq\mytilde{\Omega}_k\subseteq B$ and $m(\mytilde{\Omega}_k)\le 2m(\mybar{\Omega_k})=2m_k$, then $u_k\in\calF_{\mytilde{\Omega}_k}$. By \hyperlink{eqn_FKPsinu}{FK($\Psi,\nu$)}, we have
$$\lambda_1(\mytilde{\Omega}_k)\ge\frac{C_F}{\Psi(r)}\left(\frac{m(B)}{m(\mytilde{\Omega}_k)}\right)^\nu\ge\frac{C_F}{\Psi(r)}\left(\frac{m(B)}{2m_k}\right)^\nu,$$
hence
$$\calE(u_k,u_k)\ge\frac{C_F}{\Psi(r)}\left(\frac{m(B)}{2m_k}\right)^\nu\mynorm{u_k}^2_2,$$
hence
$$\calE(u,u)=\sum_{k\in\bbZ}\calE(u_k,u_k)\ge\sum_{k\in\bbZ}\frac{C_F}{\Psi(r)}\left(\frac{m(B)}{2m_k}\right)^\nu\mynorm{u_k}_2^2,$$
where
$$\mynorm{u_k}_2^2=\int_Xu_k^2\md m\ge\int_{\mybar{\Omega_{k+1}}}u_k^2\md m\ge4^km_{k+1},$$
hence
$$\calE(u,u)\ge\frac{C_Fm(B)^\nu}{2^\nu\Psi(r)}\sum_{k\in\bbZ}\frac{4^km_{k+1}}{m_k^\nu}.$$

Let $r,s>1$ satisfy $1/r+1/s=1$. For any positive sequences $\myset{x_k}$, $\myset{y_k}$, we have
$$\sum x_k\le\left(\sum x_k^{1/r}\right)^r\le\left(\sum\frac{x_k}{y_k}\right)\left(\sum y_k^{\frac{s}{r}}\right)^{\frac{r}{s}}.$$
Hence for any $\alpha\in(0,+\infty)$, we have
$$\sum x_k\le\left(\sum\frac{x_k}{y_k}\right)\left(\sum y_k^{\alpha}\right)^{\frac{1}{\alpha}}.$$
Therefore,
\begin{align*}
&\sum_{k\in\bbZ}4^{pk}m_k=4^p\sum_{k\in\bbZ}4^{pk}m_{k+1}\\
&\le4^p\left(\sum_{k\in\bbZ}\frac{4^{pk}m_{k+1}}{4^{(p-1)k}m_k^\nu}\right)\left(\sum_{k\in\bbZ}\left(4^{(p-1)k}m_k^\nu\right)^\alpha\right)^{\frac{1}{\alpha}}\\
&=4^p\left(\sum_{k\in\bbZ}\frac{4^{k}m_{k+1}}{m_k^\nu}\right)\left(\sum_{k\in\bbZ}4^{\alpha(p-1)k}m_k^{\alpha\nu}\right)^{\frac{1}{\alpha}}.
\end{align*}
Take $\alpha=\frac{1}{\nu}$ and $p:=\frac 1{1-\nu}$, so that $\alpha(p-1)=p$. Then
$$\sum_{k\in\bbZ}4^{pk}m_k\le4^p\left(\sum_{k\in\bbZ}\frac{4^{k}m_{k+1}}{m_k^\nu}\right)\left(\sum_{k\in\bbZ}4^{pk}m_k\right)^{1-\frac{1}{p}},$$
hence
$$\left(\sum_{k\in\bbZ}4^{pk}m_k\right)^{\frac{1}{p}}\le4^p\sum_{k\in\bbZ}\frac{4^{k}m_{k+1}}{m_k^\nu},$$
so
$$\int_{B}u^{2p}\md m\le4^p\sum_{k\in\bbZ}4^{pk}m_k\le4^p\left(4^p\sum_{k\in\bbZ}\frac{4^{k}m_{k+1}}{m_k^\nu}\right)^{p}\le4^p\left(4^p\frac{2^\nu\Psi(r)}{C_Fm(B)^\nu}\calE(u,u)\right)^{p},$$
that is,
$$\left(\dashint_{B}u^{2p}\md m\right)^{\frac{1}{2p}}\le\frac{2^{p+\frac{3}{2}-\frac{1}{2p}}}{\sqrt{C_F}}\sqrt{\Psi(r)}\left(\frac{1}{m(B)}\calE(u,u)\right)^{1/2}.$$
Letting $q=2p=\frac{2}{1-\nu}$, we have the desired result with $C_L=\frac{2^{\frac{q}{2}+\frac{3}{2}-\frac{1}{q}}}{\sqrt{C_F}}$.

Secondly, let $u\in\calF\cap C_c(B)$. Then $|u|\in\calF\cap C_c(B)$ is non-negative. By the first case, we have
\begin{align*}
\left(\dashint_{B}|u|^q\md m\right)^{1/q}&\le C_L\sqrt{\Psi(r)}\left(\frac{1}{m(B)}\calE(|u|,|u|)\right)^{1/2}\\
&\le C_L\sqrt{\Psi(r)}\left(\frac{1}{m(B)}\calE(u,u)\right)^{1/2}.
\end{align*}

Thirdly, let $u\in\calF_B$. Then there exists $\myset{u_k}\subseteq\calF\cap C_c(B)$ which is $\calE_1$-convergent to $u$, therefore there exists a subsequence still denoted by $\myset{u_k}$ that converges to $u$ $m$-a.e., hence
\begin{align*}
&\left(\dashint_{B}|u|^q\md m\right)^{1/q}\le\varliminf_{k\to+\infty}\left(\dashint_{B}|u_k|^q\md m\right)^{1/q}\\
&\le\varliminf_{k\to+\infty}C_L\sqrt{\Psi(r)}\left(\frac{1}{m(B)}\calE(u_k,u_k)\right)^{1/2}\\
&=C_L\sqrt{\Psi(r)}\left(\frac{1}{m(B)}\calE(u,u)\right)^{1/2}.
\end{align*}
\end{proof}

Let $U$, $V$ be two open subsets of $X$ satisfying $U\subseteq\mybar{U}\subseteq V$. We say that $\vphi\in\calF$ is a cutoff function for $U\subseteq V$ if $0\le\vphi\le1$ $m$-a.e., $\vphi=1$ $m$-a.e. in an open neighborhood of $\mybar{U}$ and $\mysupp(\vphi)\subseteq V$, where $\mysupp(f)$ refers to the support of the measure $|f|\md m$ for any given function $f$.

We say that the cutoff Sobolev inequality \ref{eqn_CS} holds if there exists $C_S\in(0,+\infty)$ such that for any $x\in X$, for any $R,r\in(0,+\infty)$, there exists a cutoff function $\vphi\in\calF$ for $B(x,R)\subseteq B(x,R+r)$ such that for any $f\in\calF$, we have
\begin{align*}
&\int_{B(x,R+r)\backslash\mybar{B(x,R)}}f^2\md\Gamma(\vphi,\vphi)\\
&\le\frac{1}{8}\int_{B(x,R+r)\backslash\mybar{B(x,R)}}\vphi^2\md\Gamma(f,f)+\frac{C_S}{\Psi(r)}\int_{B(x,R+r)\backslash\mybar{B(x,R)}}f^2\md m.\label{eqn_CS}\tag*{CS($\Psi$)}
\end{align*}

We say that the Poincar\'e inequality \ref{eqn_PI} holds if there exists $C_P\in(0,+\infty)$ such that for any ball $B=B(x,r)$, for any $u\in\calF$, we have
\begin{equation*}\label{eqn_PI}\tag*{PI($\Psi$)}
\int_{B}|u-u_B|^2\md m\le C_P\Psi(r)\int_{2B}\md\Gamma(u,u),
\end{equation*}
where $u_A$ is the mean value of $u$ on a measurable set $A$ with $m(A)\in(0,+\infty)$, that is,
$$u_A=\dashint_Au\md m=\frac{1}{m(A)}\int_Au\md m.$$

\subsection{Heat Kernel Estimates}\label{subsec_heat}

Consider the regular Dirichlet form $(\calE,\calF)$ on $L^2(X;m)$. Let $\myset{P_t}$ be the corresponding heat semi-group. Let $\myset{X_t,t\ge0,\bbP_x,x\in X\backslash\calN_0}$ be the corresponding Hunt process, where $\calN_0$ is a properly exceptional set, that is, $m(\calN_0)=0$ and $\bbP_x(X_t\in\calN_0\text{ for some }t>0)=0$ for any $x\in X\backslash\calN_0$. For any bounded Borel function $f$, we have $P_tf(x)=\bbE_xf(X_t)$ for any $t>0$, for any $x\in X\backslash\calN_0$.

The heat kernel $p_t(x,y)$ associated with the heat semi-group $\myset{P_t}$ is a measurable function defined on $(0,+\infty)\times(X\backslash\calN_0)\times(X\backslash\calN_0)$ satisfying that:
\begin{itemize}
\item For any bounded Borel function $f$, for any $t>0$, for any $x\in X\backslash\calN_0$, we have
$$P_tf(x)=\int_{X\backslash\calN_0}p_t(x,y)f(y)m(\md y).$$
\item For any $t,s>0$, for any $x,y\in X\backslash\calN_0$, we have
$$p_{t+s}(x,y)=\int_{X\backslash\calN_0}p_t(x,z)p_s(z,y)m(\md z).$$
\item For any $t>0$, for any $x,y\in X\backslash\calN_0$, we have $p_t(x,y)=p_t(y,x)$.
\end{itemize}
See \cite{GT12} for more details.

We say that the heat kernel upper bound \ref{eqn_UHK} holds if there exist a properly exceptional set $\calN$ and $C_1$, $C_2\in(0,+\infty)$ such that for any $t\in(0,+\infty)$, for any $x,y\in X\backslash\calN$, we have
\begin{equation*}\label{eqn_UHK}\tag*{UHK($\Psi$)}
p_t(x,y)\le\frac{1}{V\left(x,\Psi^{-1}(C_1t)\right)}\exp\left(-\Upsilon\left(C_2d(x,y),t\right)\right),
\end{equation*}
where
$$\Upsilon(R,t)=\sup_{s\in(0,+\infty)}\left(\frac{R}{s}-\frac{t}{\Psi(s)}\right)\asymp
\begin{cases}
\frac{R^2}{t},&\text{if }t<R,\\
\left(\frac{R}{t^{1/\beta}}\right)^{\frac{\beta}{\beta-1}},&\text{if }t\ge R.
\end{cases}$$
Then \ref{eqn_UHK} can also be re-written as follows:
$$p_t(x,y)\le
\begin{cases}
\frac{1}{V\left(x,\Psi^{-1}(C_1t)\right)}\exp\left(-C_2\frac{d(x,y)^2}{t}\right),&\text{if }t<d(x,y),\\
\frac{1}{V\left(x,\Psi^{-1}(C_1t)\right)}\exp\left(-C_2\left(\frac{d(x,y)}{t^{1/\beta}}\right)^{\frac{\beta}{\beta-1}}\right),&\text{if }t\ge d(x,y).
\end{cases}
$$
If a lower bound, similar to \ref{eqn_UHK}, with different constants $C_i$ also holds, then we say that \hypertarget{eqn_HKPsi}{HK($\Psi$)} holds. Note that, in \cite{GHL15}, a lower bound for $p_t$, called the near-diagonal lower bound \ref{eqn_NLE}, is written as
\begin{equation*}\label{eqn_NLE}\tag*{$(NLE)_{\Psi}$}
p_t(x,y)\ge\frac{c}{V(x,\Psi^{-1}(t))}
\end{equation*}
for any $t>0$ and $m$-a.e. $x,y\in X$ such that $d(x,y)\le \varepsilon \Psi^{-1}(t)$ where $c,\varepsilon>0$ are constants independent of $t,x,y$. But \cite[THEOREM 3.2]{BGK12} ensures that, if the metric $d$ is furthermore assumed to be geodesic, which is the case of cable systems, then the conjunction of \ref{eqn_UHK} and \ref{eqn_NLE} is equivalent to \hyperlink{eqn_HKPsi}{HK($\Psi$)}. 

Observe that the function $\beta\mapsto\left(\frac{d(x,y)}{t^{1/\beta}}\right)^{\frac{\beta}{\beta-1}}$ is monotone decreasing if $d(x,y)>t$ and monotone increasing if $d(x,y)\le t$. Indeed, 
$$
\log \left(\frac{d(x,y)}{t^{1/\beta}}\right)^{\frac{\beta}{\beta-1}}=\log d(x,y)+\frac 1{\beta-1}\log\frac{d(x,y)}t.
$$
Assume now that $d(x,y)\le t$, so that
$$
p_t(x,y)\le\frac{1}{V\left(x,\Psi^{-1}(C_1t)\right)}\exp\left(-C_2\left(\frac{d(x,y)}{t^{1/\beta}}\right)^{\frac{\beta}{\beta-1}}\right).
$$
Since $\beta\ge 2$, the aforementioned monotonicity therefore yields 
$$
\frac{d(x,y)^2}{t}\le\left(\frac{d(x,y)}{t^{1/\beta}}\right)^{\frac{\beta}{\beta-1}},
$$
which implies
$$
 p_t(x,y)\le\frac{1}{V\left(x,\Psi^{-1}(C_1t)\right)} \exp\left(-C_2\frac{d(x,y)^2}{t}\right). 
$$
Assume now that $t<d(x,y)$. Then, arguing similarly,
$$
\left(\frac{d(x,y)}{t^{1/\beta}}\right)^{\frac{\beta}{\beta-1}}\le \frac{d(x,y)^2}{t},
$$
therefore
$$
p_t(x,y)\le \frac{1}{V\left(x,\Psi^{-1}(C_1t)\right)} \exp\left(- C_2 \left(\frac{d(x,y)}{t^{1/\beta}}\right)^{\frac{\beta}{\beta-1}}\right).
$$
Thus, \ref{eqn_UHK} implies that
\begin{align}
p_t(x,y)&\le
\begin{cases}
\frac{1}{V\left(x,\Psi^{-1}(C_1t)\right)}\exp\left(-C_2\frac{d(x,y)^2}{t}\right),&\text{if }t\in(0,1),\\
\frac{1}{V\left(x,\Psi^{-1}(C_1t)\right)}\exp\left(-C_2\left(\frac{d(x,y)}{t^{1/\beta}}\right)^{\frac{\beta}{\beta-1}}\right),&\text{if }t\in[1,+\infty),
\end{cases}\nonumber\\
&\asymp
\begin{cases}
\frac{C_1}{V\left(x,\sqrt{t}\right)}\exp\left(-C_2\frac{d(x,y)^2}{t}\right),&\text{if }t\in(0,1),\\
\frac{C_1}{V\left(x,t^{1/\beta}\right)}\exp\left(-C_2\left(\frac{d(x,y)}{t^{1/\beta}}\right)^{\frac{\beta}{\beta-1}}\right),&\text{if }t\in[1,+\infty).
\end{cases}\label{eqn_hk}
\end{align}

One can characterize \ref{eqn_UHK} and \hyperlink{eqn_HKPsi}{HK($\Psi$)} in terms of functional inequalities as follows:

\begin{myprop}(\cite[Theorem 1.12]{AB15})\label{prop_UHK}
Let $(X,d,m,\calE,\calF)$ be an unbounded MMD space satisfying \ref{eqn_VD}. Then the followings are equivalent.
\begin{enumerate}[(1)]
\item \ref{eqn_FK} and \ref{eqn_CS}.
\item \ref{eqn_UHK}.
\end{enumerate}
\end{myprop}

\begin{myprop}(\cite[THEOREM 1.2]{GHL15})\label{prop_HK}
Let $(X,d,m,\calE,\calF)$ be an unbounded geodesic MMD space satisfying \ref{eqn_VPhi}. Then the followings are equivalent.
\begin{enumerate}[(1)]
\item \ref{eqn_PI} and \ref{eqn_CS}.
\item \hyperlink{eqn_HKPsi}{HK($\Psi$)}.
\end{enumerate}
\end{myprop}

\begin{myrmk}
{\em 
On any complete non-compact Riemannian manifold, \ref{eqn_CS} with $\beta=2$ (that is, $\Psi(r)=r^2$ for any $r\in(0,+\infty)$) holds automatically, so that the above equivalences hold without \ref{eqn_CS} and are classical, see \cite{Gri92,Sal92,Gri94}. However, on a general MMD space, \ref{eqn_CS} does not always hold and is involved in the formulation of the previous equivalences. Moreover, \ref{eqn_CS} is directly used in the present paper to provide an $L^1$-mean value inequality in Lemma \ref{lem_mv} below.
}
\end{myrmk}

\subsection{The Poisson Equation}\label{subsec_Poi}

Let $D$ be an open subset of $X$. Let $f\in L^1_{loc}(D)$. We say that $u\in\calF$ is a solution of the Poisson equation $\Delta u=f$ in $D$ if
$$\calE(u,\vphi)=\int_Df\vphi\md m\text{ for any }\vphi\in\calF\cap C_c(D).$$
If $\Delta u=f$ in $D$ with $f\in L^2(D)$, then the above equation also holds for any $\vphi\in\calF_D$. We say that $u\in\calF$ is harmonic in $D$ if $\Delta u=0$ in $D$.

We have some results about the existence, the uniqueness and the regularity of the solutions of Poisson equation, that we now state and prove.

\begin{mylem}\label{lem_Poi_exist}
Let $(X,d,m,\calE,\calF)$ be an unbounded MMD space satisfying \hyperlink{eqn_LSPsiq}{LS($\Psi,q$)}. Then for any $p\in\left[\frac{q}{q-1},+\infty\right)$, for any ball $B=B(x_0,r)$, for any $f\in L^p(B)$, there exists a unique \footnote{in the sense that if $u_1,u_2\in\calF_B$ satisfy $\Delta u_1=\Delta u_2=f$ in $B$, then $u_1=u_2$ $m$-a.e..} $u\in\calF_B$ such that $\Delta u=f$ in $B$. There exists $C\in(0,+\infty)$ depending only on $C_L$ such that
$$\dashint_B|u|\md m\le C\sqrt{\Psi(r)}\left(\frac{1}{m(B)}\calE(u,u)\right)^{1/2}\le C\Psi(r)\left(\dashint_B|f|^p\md m\right)^{1/p}.$$
\end{mylem}

\begin{proof}
The proof is inspired by \cite[Lemma 2.6]{CJKS20}. First, we prove the existence. By \hyperlink{eqn_LSPsiq}{LS($\Psi,q$)}, for any $\vphi\in\calF_{B}$, we have
\begin{align*}
&\left(\int_{B}|\vphi|^2\md m\right)^{1/2}\le m(B)^{1/2}\left(\dashint_{B}|\vphi|^q\md m\right)^{1/q}\\
&\le m(B)^{1/2}C_L\sqrt{\Psi(r)}\left(\frac{1}{m(B)}\calE(\vphi,\vphi)\right)^{1/2}=C_L\sqrt{\Psi(r)}\calE(\vphi,\vphi)^{1/2},
\end{align*}
hence $(\calF_{B},\calE)$ is a Hilbert space.

We split the rest of the proof of the existence part into two steps. To start with, we assume that $f\in L^2(B)$. For any $\vphi\in\calF_B$, since
\begin{align*}|
&\int_{B}f\vphi\md m|\le\left(\int_{B}|f|^2\md m\right)^{1/2}\left(\int_{B}|\vphi|^2\md m\right)^{1/2}\le \left(\int_{B}|f|^2\md m\right)^{1/2}C_L\sqrt{\Psi(r)}\calE(\vphi,\vphi)^{1/2},
\end{align*}
we have $\vphi\mapsto\int_{B}f\vphi\md m$ is a bounded linear functional on $(\calF_{B},\calE)$. By the Riesz representation theorem, there exists a unique $u\in\calF_{B}$ such that $\calE(u,\vphi)=\int_{B}f\vphi\md m$ for any $\vphi\in\calF_{B}$, hence $\Delta u=f$ in $B$.

Next, we assume that $f\in L^p(B)$. For any $k\ge1$, let $f_k=(f\vee(-k))\wedge k$, then $f_k\in L^\infty(B)\subseteq L^2(B)$ and $\{f_k\}$ converges to $f$ in $L^p(B)$. By the first step, there exists $u_k\in\calF_B$ such that $\Delta u_k=f_k$ in $B$. For any $k,l\ge1$, by \hyperlink{eqn_LSPsiq}{LS($\Psi,q$)}, we have
\begin{align*}
&\calE(u_k-u_l,u_k-u_l)=\int_B(f_k-f_l)(u_k-u_l)\md m\\
&\le\mynorm{f_k-f_l}_{L^p(B)}\left(\dashint_B|u_k-u_l|^{p'}\md m\right)^{1/p'}m(B)^{1/p'}\\
&\le\mynorm{f_k-f_l}_{L^p(B)}\left(\dashint_B|u_k-u_l|^{q}\md m\right)^{1/q}m(B)^{1/p'}\\
&\le\mynorm{f_k-f_l}_{L^p(B)}C_L\sqrt{\Psi(r)}\left(\frac{1}{m(B)}\calE(u_k-u_l,u_k-u_l)\right)^{1/2}m(B)^{1/p'},
\end{align*}
where the third line uses the inequality $p^{\prime}\le q$. Hence
\begin{equation}\label{eqn_similar}
\calE(u_k-u_l,u_k-u_l)^{1/2}\le C_L\sqrt{\Psi(r)}m(B)^{\frac{1}{2}-\frac{1}{p}}\mynorm{f_k-f_l}_{L^p(B)},
\end{equation}
hence $\myset{u_k}$ is an $\calE$-Cauchy sequence in $\calF_B$. Since $(\calF_B,\calE)$ is a Hilbert space, there exists $u\in\calF_B$ such that $\myset{u_k}$ is $\calE$-convergent to $u$.

For any $\vphi\in\calF_B$, we have
$$\calE(u,\vphi)=\lim_{k\to+\infty}\calE(u_k,\vphi)=\lim_{k\to+\infty}\int_{B}f_k\vphi\md m.$$
Since $\vphi\in\calF_B$, by \hyperlink{eqn_LSPsiq}{LS($\Psi,q$)}, we have $\vphi\in L^q(B)$, so that $\vphi\in L^{p^{\prime}}(B)$ (recall that $p^{\prime}\le q$). Using the convergence of $\myset{f_k}$ to $f$ in $L^p(B)$, we have
$$\lim_{k\to+\infty}\int_{B}f_k\vphi\md m=\int_Bf\vphi\md m.$$
Hence $\calE(u,\vphi)=\int_Bf\vphi\md m$ for any $\vphi\in\calF_B$, hence $\Delta u=f$ in $B$. This concludes the proof of the existence.

We now prove the $L^1$-estimate. Let $u\in\calF_B$ satisfy $\Delta u=f$ in $B$. Similarly to Equation (\ref{eqn_similar}), we have
$$\calE(u,u)^{1/2}\le C_L\sqrt{\Psi(r)}m(B)^{\frac{1}{2}-\frac{1}{p}}\mynorm{f}_{L^p(B)}.$$
By \hyperlink{eqn_LSPsiq}{LS($\Psi,q$)}, we have
\begin{align*}
&\dashint_B|u|\md m\le\left(\dashint_B|u|^q\md m\right)^{1/q}\le C_L\sqrt{\Psi(r)}\left(\frac{1}{m(B)}\calE(u,u)\right)^{1/2}\\
&\le C_L\sqrt{\Psi(r)}\frac{1}{m(B)^{1/2}}C_L\sqrt{\Psi(r)}m(B)^{\frac{1}{2}-\frac{1}{p}}\mynorm{f}_{L^p(B)}=C_L^2\Psi(r)\left(\dashint_B|f|^p\md m\right)^{1/p}.
\end{align*}

Finally, we prove the uniqueness. Indeed, let $u_1,u_2\in\calF_B$ satisfy $\Delta u_1=\Delta u_2=f$ in $B$, then $u_1-u_2\in\calF_B$ satisfies $\Delta(u_1-u_2)=0$ in $B$. By the above $L^1$-estimate, we have $u_1=u_2$ $m$-a.e..
\end{proof}

\begin{mylem}\label{lem_Poi_point}
Let $(X,d,m,\calE,\calF)$ be an unbounded MMD space satisfying \ref{eqn_VD}, \hyperlink{eqn_LSPsiq}{LS($\Psi,q$)} and \ref{eqn_CS}. Then for any $p\in\left[\frac{q}{q-1},+\infty\right)$, there exists $C\in(0,+\infty)$ such that for any ball $B=B(x_0,r)$, for any $f\in L^\infty(2B)$, if $u\in\calF$ satisfies $\Delta u=f$ in $2B$, then for $m$-a.e. $x\in B$, we have
$$|u(x)|\le C\left(\dashint_{2B}|u|\md m+F_1(x)\right),$$
where
$$F_1(x)=\sum_{j\le[\log_2r]}\Psi(2^j)\left(\dashint_{B(x,2^j)}|f|^p\md m\right)^{1/p}.$$ 
\end{mylem}

The proof is inspired by \cite[Proposition 3.1]{CJKS20}, where an $L^1$-version of the mean value inequality (see \cite[Proposition 2.1]{CJKS20}) was needed. The condition \ref{eqn_CS} is intrinsically used to obtain the $L^1$-mean value inequality as follows.

\begin{mylem}(\cite[THEOREM 6.3, LEMMA 9.2]{GHL15})\label{lem_mv}
Let $(X,d,m,\calE,\calF)$ be an unbounded MMD space satisfying \ref{eqn_VD}, \ref{eqn_LS} and \ref{eqn_CS}. Then there exists $C\in(0,+\infty)$ such that for any ball $B=B(x_0,r)$, for any $u\in\calF$ which is harmonic in $2B$, we have
$$\mynorm{u}_{L^\infty\left(B\right)}\le C\dashint_{2B}|u|\md m.$$
\end{mylem}

\begin{proof}[Proof of Lemma \ref{lem_Poi_point}]
Let $j_0=[\log_2r]$. Take an arbitrary Lebesgue point $x\in B$ of $u\in\calF$. For any $j\le j_0$, by Lemma \ref{lem_Poi_exist}, there exists $u_j\in\calF_{B(x,2^j)}$ such that $\Delta u_j=f$ in $B(x,2^j)$ and
$$\dashint_{B(x,2^{j-1})}|u_j|\md m\lesssim\dashint_{B(x,2^j)}|u_j|\md m\lesssim\Psi(2^j)\left(\dashint_{B(x,2^j)}|f|^p\md m\right)^{1/p}.$$
Since $\Delta(u_j-u_{j-1})=0$ in $B(x,2^{j-1})$, by Lemma \ref{lem_mv}, we have
\begin{align*}
&\mynorm{u_j-u_{j-1}}_{L^\infty(B(x,2^{j-2}))}\lesssim\dashint_{B(x,2^{j-1})}|u_j-u_{j-1}|\md m\\
&\le\dashint_{B(x,2^{j-1})}|u_j|\md m+\dashint_{B(x,2^{j-1})}|u_{j-1}|\md m\\
&\lesssim\Psi(2^{j})\left(\dashint_{B(x,2^j)}|f|^p\md m\right)^{1/p}+\Psi(2^{j-1})\left(\dashint_{B(x,2^{j-1})}|f|^p\md m\right)^{1/p}.
\end{align*}
Since $\Delta(u-u_{j_0})=0$ in $B(x,2^{j_0})$, by Lemma \ref{lem_mv}, we have
\begin{align*}
&\mynorm{u-u_{j_0}}_{L^\infty(B(x,2^{j_0-1}))}\lesssim\dashint_{B(x,2^{j_0})}|u-u_{j_0}|\md m\\
&\le\dashint_{B(x,2^{j_0})}|u|\md m+\dashint_{B(x,2^{j_0})}|u_{j_0}|\md m\lesssim\dashint_{2B}|u|\md m+\Psi(2^{j_0})\left(\dashint_{B(x,2^{j_0})}|f|^p\md m\right)^{1/p}.
\end{align*}
Hence
\begin{align*}
&|u(x)|=\lim_{k\to-\infty}\dashint_{B(x,2^{k})}|u|\md m\\
&\le\varliminf_{k\to-\infty}\dashint_{B(x,2^{k})}\left(|u-u_{j_0}|+\sum_{j=k+2}^{j_0}|u_j-u_{j-1}|+|u_{k+1}|\right)\md m\\
&\le\varliminf_{k\to-\infty}\left(\mynorm{u-u_{j_0}}_{L^\infty(B(x,2^{j_0-1}))}+\sum_{j=k+2}^{j_0}\mynorm{u_j-u_{j-1}}_{L^\infty(B(x,2^{j-2}))}+\dashint_{B(x,2^k)}|u_{k+1}|\md m\right)\\
&\lesssim\varliminf_{k\to-\infty}\left(\dashint_{2B}|u|\md m+\Psi(2^{j_0})\left(\dashint_{B(x,2^{j_0})}|f|^p\md m\right)^{1/p}\right.\\
&+\sum_{j=k+2}^{j_0}\left(\Psi(2^{j})\left(\dashint_{B(x,2^j)}|f|^p\md m\right)^{1/p}+\Psi(2^{j-1})\left(\dashint_{B(x,2^{j-1})}|f|^p\md m\right)^{1/p}\right)\\
&\left.+\Psi(2^{k+1})\left(\dashint_{B(x,2^{k+1})}|f|^p\md m\right)^{1/p}\right)\\
&\lesssim\varliminf_{k\to-\infty}\left(\dashint_{2B}|u|\md m+\sum_{j=k+1}^{j_0}\Psi(2^{j})\left(\dashint_{B(x,2^j)}|f|^p\md m\right)^{1/p}\right)\\
&=\dashint_{2B}|u|\md m+\sum_{j\le j_0}\Psi(2^{j})\left(\dashint_{B(x,2^j)}|f|^p\md m\right)^{1/p}.
\end{align*}
\end{proof}

Let us end up this section by presenting reverse H\"older inequalities. We say that an MMD space $(X,d,m,\calE,\calF)$ admits a ``carr\'e du champ" if the energy measure $\Gamma(u,v)$ is absolutely continuous with respect to $m$ for any $u,v\in\calF$. Let $\langle\nabla u,\nabla v\rangle$ denote the Radon derivative $\frac{\md\Gamma(u,v)}{\md m}$ and let $|\nabla u|$ denote the square root of the Radon derivative $\frac{\md\Gamma(u,u)}{\md m}$.

As already encountered in the introduction, say that the reverse H\"older inequality \ref{eqn_RH} holds if there exists $C_H\in(0,+\infty)$ such that for any ball $B=B(x_0,r)$, for any $u\in\calF$ which is harmonic in $2B$, we have
\begin{equation*}\label{eqn_RH}\tag*{RH}
\mynorm{|\nabla u|}_{L^\infty(B)}\le\frac{C_H}{r}\dashint_{2B}|u|\md m.
\end{equation*}

We say that the generalized reverse H\"older inequality \ref{eqn_GRH} holds if there exists $C_H\in(0,+\infty)$ such that for any ball $B=B(x_0,r)$, for any $u\in\calF$ which is harmonic in $2B$, we have
\begin{equation*}\label{eqn_GRH}\tag*{GRH($\Phi,\Psi$)}
\mynorm{|\nabla u|}_{L^\infty(B)}\le C_H\frac{\Phi(r)}{\Psi(r)}\dashint_{2B}|u|\md m,
\end{equation*}
or equivalently, 
$$\mynorm{|\nabla u|}_{L^\infty(B)}\le
\begin{cases}
\frac{C_H}{r}\dashint_{2B}|u|\md m,&\text{if }r\in(0,1),\\
\frac{C_H}{r^{\beta-\alpha}}\dashint_{2B}|u|\md m,&\text{if }r\in[1,+\infty).
\end{cases}
$$

\section{The Vicsek and the Sierpi\'nski Cable Systems}\label{sec_cable}

Let $(V,E)$ be an infinite, locally bounded, connected (undirected) graph, that is, $V$ is the set of vertices which is a countably infinite set, $E\subseteq\{\{p,q\}:p,q\in V\}$ is the set of edges satisfying $\{p,q\}\in E$ if and only if $\{q,p\}\in E$, $\sup_{p\in V}\#\{q\in V:\{p,q\}\in E\}<+\infty$ and for any distinct $p,q\in V$, there exist an integer $n\ge1$ and $p_0,p_1,\ldots,p_n\in V$ satisfying $p_0=p$, $p_n=q$ and $\{p_i,p_{i+1}\}\in E$ for any $i=0,\ldots,n-1$.

We give an arbitrary orientation on each edge $e\in E$ by taking $s:E\to V$ and $t:E\to V$ such that $e=\{s(e),t(e)\}$. Let
$$X=(E\times[0,1])/\sim,$$
where $\sim$ is an equivalence relation given by $s(e_1)=s(e_2)$ implies $(e_1,0)\sim(e_2,0)$, $t(e_1)=t(e_2)$ implies $(e_1,1)\sim(e_2,1)$ and $s(e_1)=t(e_2)$ implies $(e_1,0)\sim (e_2,1)$ for any $e_1,e_2\in E$. Let $\pi:E\times[0,1]\to X$ be the quotient map. We have $V=\pi(E\times\{0,1\})\subseteq X$. For any $x=\pi(e,a)$ and $y=\pi(e,b)$ with $e\in E$ and $a,b\in[0,1]$, let
\begin{align*}
[x,y]&:=\pi\left(e\times[\min\{a,b\},\max\{a,b\}]\right),\\
(x,y)&:=\pi\left(e\times(\min\{a,b\},\max\{a,b\})\right).
\end{align*}
For any $\{p,q\}\in E$, we say that $[p,q]$ is a closed cable and $(p,q)$ is an open cable.

For any distinct $p,q\in V\subseteq X$, let $d(p,p)=0$ and
$$d(p,q)=\inf\myset{n:p=p_0,p_1,\ldots,p_n=q\in V,\{p_i,p_{i+1}\}\in E\text{ for any }i=0,\ldots,n-1}.$$
For any $x,y\in X$, if there exist $e\in E$ and $a,b\in[0,1]$ such that $x=\pi(e,a)$ and $y=\pi(e,b)$, then let $d(x,y)=|a-b|$. Otherwise there exist distinct $e_1,e_2\in E$, there exist $a,b\in[0,1]$ such that $x=\pi(e_1,a)$ and $y=\pi(e_2,b)$, let
\begin{align*}
d(x,y)=\min\Bigl\{&|a|+d(\pi(e_1,0),\pi(e_2,0))+|b|,|a|+d(\pi(e_1,0),\pi(e_2,1))+|b-1|,\\
&|a-1|+d(\pi(e_1,1),\pi(e_2,0))+|b|,|a-1|+d(\pi(e_1,1),\pi(e_2,1))+|b-1|\Bigr\}.
\end{align*}
It is obvious that $d$ is well-defined and $(X,d)$ is a locally compact separable unbounded geodesic metric space. Let $m$ be the unique positive Radon measure on $X$ satisfying
$$m(\pi(e\times[a,b]))=|a-b|\text{ for any }e\in E,\text{ for any }a,b\in[0,1]\text{ with }a\le b.$$

Let $u$ and $v$ be two real-valued functions on $X$, and let $p,q\in V$ with $\{p,q\}\in E$. For any $x$ in the open cable $(p,q)$, define
$$\nabla u(x)=\lim_{(p,q)\ni y\to x}\frac{u(y)-u(x)}{d(y,p)-d(x,p)}.$$
At the vertex $p$ itself, we define the directional derivative in the direction $q$ as
$$\nabla_qu(p)=\lim_{(p,q)\ni y\to p}\frac{u(y)-u(p)}{d(y,p)}.$$
Note that the choice of the roles of $p,q$ determines the sign of $\nabla u(x)$ but does not influence $|\nabla u(x)|$ and $\nabla u(x)\nabla v(x)$. For any measurable subset $D$ of $X$, we denote
$$\mynorm{|\nabla u|}_{L^\infty(D)}=\esssup_{x\in D\backslash V}|\nabla u(x)|.$$
Note that $m(V)=0$, so the above definition makes sense even if $\nabla u(x)$ is not well-defined for any $x\in V$.

Let
\begin{align*}
\calK=&\left\{u\in C_c(X):\nabla u(x),\nabla_qu(p)\text{ exist for any }x\in(p,q),\right.\\
&\left.\text{ for any }p,q\in V\text{ with }\{p,q\}\in E,\mynorm{|\nabla u|}_{L^\infty(X;m)}<+\infty\right\}.
\end{align*}
Let
\begin{align*}
\calE(u,u)&=\frac{1}{2}
\sum_{\mbox{\tiny
$
\begin{subarray}{c}
p,q\in V\\
\{p,q\}\in E
\end{subarray}
$
}}\int_{(p,q)}|\nabla u|^2\md m,\\
\calF&=\text{ the }\calE_1\text{-closure of }\calK.
\end{align*}
Then $(\calE,\calF)$ is a strongly local regular Dirichlet form on $L^2(X;m)$, $(X,d,m,\calE,\calF)$ is an unbounded geodesic MMD space called an unbounded cable system.

It is obvious that $(X,d,m,\calE,\calF)$ admits a ``carr\'e du champ". Indeed, for any $u,v\in\calF$, $\nabla u\nabla v$ is the Radon derivative $\frac{\md\Gamma(u,v)}{\md m}$ and $|\nabla u|^2$ is the Radon derivative $\frac{\md\Gamma(u,u)}{\md m}$. 

Harmonic functions have the following explicit characterization. Let $D$ be a domain in $X$, that is, $D$ is a connected open subset of $X$. Let $u\in\calF$. Then $u$ is harmonic in $D$ if and only if:
\begin{itemize}
\item For any open cable $(p,q)$ intersecting $D$, the function $u$ is linear on each open connected component of $(p,q)\cap D$ (note that there are at most two such components).
\item For any $p\in V\cap D$, the directional derivative $\nabla_qu(p)$ exists for any $q\in V$ with $\{p,q\}\in E$ and the following Kirchhoff condition at $p$ holds:
$$\sum_{\mbox{\tiny
$
\begin{subarray}{c}
q\in V\\
\{p,q\}\in E
\end{subarray}
$
}}\nabla_qu(p)=0.$$
\end{itemize}
See \cite[Section 1.3]{S94} for the Kirchhoff condition.

After the introduction of general cable systems, we describe our two main examples: the Vicsek and the Sierpi\'nski cable systems. Let us start with the Vicsek cable systems. Let $N\ge2$ be an integer. In $\R^N$, let $p_1=(0,\ldots,0),\ldots,p_{2^N}$ be the vertices of the cube $[0,\frac{2}{\sqrt{N}}]^N\subseteq\R^N$, let $p_{2^N+1}=\frac{1}{2^N}\sum_{i=1}^{2^N}p_i=(\frac{1}{\sqrt{N}},\ldots,\frac{1}{\sqrt{N}})$. Let $f_i(x)=\frac{1}{3}x+\frac{2}{3}p_i$, $x\in\R^N$, $i=1,\ldots,2^N,2^N+1$. Then the $N$-dimensional Vicsek set is the unique non-empty compact set $K$ in $\R^N$ satisfying $K=\cup_{i=1}^{2^N+1}f_i(K)$.

Let $V_0=\{p_1,\ldots,p_{2^N},p_{2^N+1}\}$ and $V_{n+1}=\cup_{i=1}^{2^N+1}f_i(V_n)$ for any $n\ge0$. Then $\myset{V_n}_{n\ge0}$ is an increasing sequence of finite subsets of $K$ and the closure of $\cup_{n\ge0}V_n$ is $K$.

\begin{figure}[ht]
\centering
\begin{subfigure}[b]{0.2\textwidth}
\centering
\begin{tikzpicture}[scale=0.15]
\draw (0,0)--(2,2);
\draw (0,2)--(2,0);
\draw[fill=black] (0,0) circle (0.25);
\draw[fill=black] (0,2) circle (0.25);
\draw[fill=black] (2,2) circle (0.25);
\draw[fill=black] (2,0) circle (0.25);
\draw[fill=black] (1,1) circle (0.25);
\end{tikzpicture}
\caption{$V^{(0)}$}
\end{subfigure}
\hspace{2em}
\begin{subfigure}[b]{0.2\textwidth}
\centering
\begin{tikzpicture}[scale=0.15]
\draw (0,2)--(2,0);
\draw (0,0)--(2,2);
\draw (4,0)--(6,2);
\draw (4,2)--(6,0);
\draw (4,6)--(6,4);
\draw (4,4)--(6,6);
\draw (0,4)--(2,6);
\draw (2,4)--(0,6);
\draw (2,2)--(4,4);
\draw (2,4)--(4,2);

\draw[fill=black] (0,0) circle (0.25);
\draw[fill=black] (0,2) circle (0.25);
\draw[fill=black] (2,2) circle (0.25);
\draw[fill=black] (2,0) circle (0.25);
\draw[fill=black] (1,1) circle (0.25);

\draw[fill=black] (0+4,0) circle (0.25);
\draw[fill=black] (0+4,2) circle (0.25);
\draw[fill=black] (2+4,2) circle (0.25);
\draw[fill=black] (2+4,0) circle (0.25);
\draw[fill=black] (1+4,1) circle (0.25);

\draw[fill=black] (0,0+4) circle (0.25);
\draw[fill=black] (0,2+4) circle (0.25);
\draw[fill=black] (2,2+4) circle (0.25);
\draw[fill=black] (2,0+4) circle (0.25);
\draw[fill=black] (1,1+4) circle (0.25);

\draw[fill=black] (0+4,0+4) circle (0.25);
\draw[fill=black] (0+4,2+4) circle (0.25);
\draw[fill=black] (2+4,2+4) circle (0.25);
\draw[fill=black] (2+4,0+4) circle (0.25);
\draw[fill=black] (1+4,1+4) circle (0.25);

\draw[fill=black] (0+2,0+2) circle (0.25);
\draw[fill=black] (0+2,2+2) circle (0.25);
\draw[fill=black] (2+2,2+2) circle (0.25);
\draw[fill=black] (2+2,0+2) circle (0.25);
\draw[fill=black] (1+2,1+2) circle (0.25);
\end{tikzpicture}
\caption{$V^{(1)}$}
\end{subfigure}
\hspace{2em}
\begin{subfigure}[b]{0.3\textwidth}
\centering
\begin{tikzpicture}[scale=0.15]
\draw (0,2)--(2,0);
\draw (0,0)--(2,2);
\draw (4,0)--(6,2);
\draw (4,2)--(6,0);
\draw (4,6)--(6,4);
\draw (4,4)--(6,6);
\draw (0,4)--(2,6);
\draw (2,4)--(0,6);
\draw (2,2)--(4,4);
\draw (2,4)--(4,2);

\draw (0+12,2)--(2+12,0);
\draw (0+12,0)--(2+12,2);
\draw (4+12,0)--(6+12,2);
\draw (4+12,2)--(6+12,0);
\draw (4+12,6)--(6+12,4);
\draw (4+12,4)--(6+12,6);
\draw (0+12,4)--(2+12,6);
\draw (2+12,4)--(0+12,6);
\draw (2+12,2)--(4+12,4);
\draw (2+12,4)--(4+12,2);

\draw (0+6,2+6)--(2+6,0+6);
\draw (0+6,0+6)--(2+6,2+6);
\draw (4+6,0+6)--(6+6,2+6);
\draw (4+6,2+6)--(6+6,0+6);
\draw (4+6,6+6)--(6+6,4+6);
\draw (4+6,4+6)--(6+6,6+6);
\draw (0+6,4+6)--(2+6,6+6);
\draw (2+6,4+6)--(0+6,6+6);
\draw (2+6,2+6)--(4+6,4+6);
\draw (2+6,4+6)--(4+6,2+6);

\draw (0,2+12)--(2,0+12);
\draw (0,0+12)--(2,2+12);
\draw (4,0+12)--(6,2+12);
\draw (4,2+12)--(6,0+12);
\draw (4,6+12)--(6,4+12);
\draw (4,4+12)--(6,6+12);
\draw (0,4+12)--(2,6+12);
\draw (2,4+12)--(0,6+12);
\draw (2,2+12)--(4,4+12);
\draw (2,4+12)--(4,2+12);

\draw (0+12,2+12)--(2+12,0+12);
\draw (0+12,0+12)--(2+12,2+12);
\draw (4+12,0+12)--(6+12,2+12);
\draw (4+12,2+12)--(6+12,0+12);
\draw (4+12,6+12)--(6+12,4+12);
\draw (4+12,4+12)--(6+12,6+12);
\draw (0+12,4+12)--(2+12,6+12);
\draw (2+12,4+12)--(0+12,6+12);
\draw (2+12,2+12)--(4+12,4+12);
\draw (2+12,4+12)--(4+12,2+12);

\draw[fill=black] (0,0) circle (0.25);
\draw[fill=black] (0,2) circle (0.25);
\draw[fill=black] (2,2) circle (0.25);
\draw[fill=black] (2,0) circle (0.25);
\draw[fill=black] (1,1) circle (0.25);

\draw[fill=black] (0+4,0) circle (0.25);
\draw[fill=black] (0+4,2) circle (0.25);
\draw[fill=black] (2+4,2) circle (0.25);
\draw[fill=black] (2+4,0) circle (0.25);
\draw[fill=black] (1+4,1) circle (0.25);

\draw[fill=black] (0,0+4) circle (0.25);
\draw[fill=black] (0,2+4) circle (0.25);
\draw[fill=black] (2,2+4) circle (0.25);
\draw[fill=black] (2,0+4) circle (0.25);
\draw[fill=black] (1,1+4) circle (0.25);

\draw[fill=black] (0+4,0+4) circle (0.25);
\draw[fill=black] (0+4,2+4) circle (0.25);
\draw[fill=black] (2+4,2+4) circle (0.25);
\draw[fill=black] (2+4,0+4) circle (0.25);
\draw[fill=black] (1+4,1+4) circle (0.25);

\draw[fill=black] (0+2,0+2) circle (0.25);
\draw[fill=black] (0+2,2+2) circle (0.25);
\draw[fill=black] (2+2,2+2) circle (0.25);
\draw[fill=black] (2+2,0+2) circle (0.25);
\draw[fill=black] (1+2,1+2) circle (0.25);

\draw[fill=black] (0+12,0) circle (0.25);
\draw[fill=black] (0+12,2) circle (0.25);
\draw[fill=black] (2+12,2) circle (0.25);
\draw[fill=black] (2+12,0) circle (0.25);
\draw[fill=black] (1+12,1) circle (0.25);

\draw[fill=black] (0+4+12,0) circle (0.25);
\draw[fill=black] (0+4+12,2) circle (0.25);
\draw[fill=black] (2+4+12,2) circle (0.25);
\draw[fill=black] (2+4+12,0) circle (0.25);
\draw[fill=black] (1+4+12,1) circle (0.25);

\draw[fill=black] (0+12,0+4) circle (0.25);
\draw[fill=black] (0+12,2+4) circle (0.25);
\draw[fill=black] (2+12,2+4) circle (0.25);
\draw[fill=black] (2+12,0+4) circle (0.25);
\draw[fill=black] (1+12,1+4) circle (0.25);

\draw[fill=black] (0+4+12,0+4) circle (0.25);
\draw[fill=black] (0+4+12,2+4) circle (0.25);
\draw[fill=black] (2+4+12,2+4) circle (0.25);
\draw[fill=black] (2+4+12,0+4) circle (0.25);
\draw[fill=black] (1+4+12,1+4) circle (0.25);

\draw[fill=black] (0+2+12,0+2) circle (0.25);
\draw[fill=black] (0+2+12,2+2) circle (0.25);
\draw[fill=black] (2+2+12,2+2) circle (0.25);
\draw[fill=black] (2+2+12,0+2) circle (0.25);
\draw[fill=black] (1+2+12,1+2) circle (0.25);

\draw[fill=black] (0,0+12) circle (0.25);
\draw[fill=black] (0,2+12) circle (0.25);
\draw[fill=black] (2,2+12) circle (0.25);
\draw[fill=black] (2,0+12) circle (0.25);
\draw[fill=black] (1,1+12) circle (0.25);

\draw[fill=black] (0+4,0+12) circle (0.25);
\draw[fill=black] (0+4,2+12) circle (0.25);
\draw[fill=black] (2+4,2+12) circle (0.25);
\draw[fill=black] (2+4,0+12) circle (0.25);
\draw[fill=black] (1+4,1+12) circle (0.25);

\draw[fill=black] (0,0+4+12) circle (0.25);
\draw[fill=black] (0,2+4+12) circle (0.25);
\draw[fill=black] (2,2+4+12) circle (0.25);
\draw[fill=black] (2,0+4+12) circle (0.25);
\draw[fill=black] (1,1+4+12) circle (0.25);

\draw[fill=black] (0+4,0+4+12) circle (0.25);
\draw[fill=black] (0+4,2+4+12) circle (0.25);
\draw[fill=black] (2+4,2+4+12) circle (0.25);
\draw[fill=black] (2+4,0+4+12) circle (0.25);
\draw[fill=black] (1+4,1+4+12) circle (0.25);

\draw[fill=black] (0+2,0+2+12) circle (0.25);
\draw[fill=black] (0+2,2+2+12) circle (0.25);
\draw[fill=black] (2+2,2+2+12) circle (0.25);
\draw[fill=black] (2+2,0+2+12) circle (0.25);
\draw[fill=black] (1+2,1+2+12) circle (0.25);

\draw[fill=black] (0+12,0+12) circle (0.25);
\draw[fill=black] (0+12,2+12) circle (0.25);
\draw[fill=black] (2+12,2+12) circle (0.25);
\draw[fill=black] (2+12,0+12) circle (0.25);
\draw[fill=black] (1+12,1+12) circle (0.25);

\draw[fill=black] (0+4+12,0+12) circle (0.25);
\draw[fill=black] (0+4+12,2+12) circle (0.25);
\draw[fill=black] (2+4+12,2+12) circle (0.25);
\draw[fill=black] (2+4+12,0+12) circle (0.25);
\draw[fill=black] (1+4+12,1+12) circle (0.25);

\draw[fill=black] (0+12,0+4+12) circle (0.25);
\draw[fill=black] (0+12,2+4+12) circle (0.25);
\draw[fill=black] (2+12,2+4+12) circle (0.25);
\draw[fill=black] (2+12,0+4+12) circle (0.25);
\draw[fill=black] (1+12,1+4+12) circle (0.25);

\draw[fill=black] (0+4+12,0+4+12) circle (0.25);
\draw[fill=black] (0+4+12,2+4+12) circle (0.25);
\draw[fill=black] (2+4+12,2+4+12) circle (0.25);
\draw[fill=black] (2+4+12,0+4+12) circle (0.25);
\draw[fill=black] (1+4+12,1+4+12) circle (0.25);

\draw[fill=black] (0+2+12,0+2+12) circle (0.25);
\draw[fill=black] (0+2+12,2+2+12) circle (0.25);
\draw[fill=black] (2+2+12,2+2+12) circle (0.25);
\draw[fill=black] (2+2+12,0+2+12) circle (0.25);
\draw[fill=black] (1+2+12,1+2+12) circle (0.25);

\draw[fill=black] (0+6,0+6) circle (0.25);
\draw[fill=black] (0+6,2+6) circle (0.25);
\draw[fill=black] (2+6,2+6) circle (0.25);
\draw[fill=black] (2+6,0+6) circle (0.25);
\draw[fill=black] (1+6,1+6) circle (0.25);

\draw[fill=black] (0+4+6,0+6) circle (0.25);
\draw[fill=black] (0+4+6,2+6) circle (0.25);
\draw[fill=black] (2+4+6,2+6) circle (0.25);
\draw[fill=black] (2+4+6,0+6) circle (0.25);
\draw[fill=black] (1+4+6,1+6) circle (0.25);

\draw[fill=black] (0+6,0+4+6) circle (0.25);
\draw[fill=black] (0+6,2+4+6) circle (0.25);
\draw[fill=black] (2+6,2+4+6) circle (0.25);
\draw[fill=black] (2+6,0+4+6) circle (0.25);
\draw[fill=black] (1+6,1+4+6) circle (0.25);

\draw[fill=black] (0+4+6,0+4+6) circle (0.25);
\draw[fill=black] (0+4+6,2+4+6) circle (0.25);
\draw[fill=black] (2+4+6,2+4+6) circle (0.25);
\draw[fill=black] (2+4+6,0+4+6) circle (0.25);
\draw[fill=black] (1+4+6,1+4+6) circle (0.25);

\draw[fill=black] (0+2+6,0+2+6) circle (0.25);
\draw[fill=black] (0+2+6,2+2+6) circle (0.25);
\draw[fill=black] (2+2+6,2+2+6) circle (0.25);
\draw[fill=black] (2+2+6,0+2+6) circle (0.25);
\draw[fill=black] (1+2+6,1+2+6) circle (0.25);

\end{tikzpicture}
\caption{$V^{(2)}$}
\end{subfigure}
\caption{$V^{(0)}$, $V^{(1)}$ and $V^{(2)}$ for $N=2$}\label{fig_V012_Vicsek}
\end{figure}
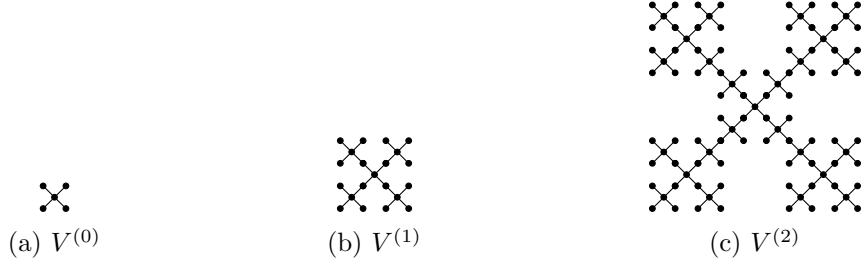

For any $n\ge0$, let $V^{(n)}=3^nV_n=\{3^nv:v\in V_n\}$, see Figure \ref{fig_V012_Vicsek} for $V^{(0)}$, $V^{(1)}$ and $V^{(2)}$ for $N=2$. Then $\myset{V^{(n)}}_{n\ge0}$ is an increasing sequence of finite sets. Let $V=\cup_{n\ge0}V^{(n)}$ and $E=\{\{p,q\}:p,q\in V,|p-q|=1\}$, then $(V,E)$ is an infinite, locally bounded, connected graph, the corresponding unbounded cable system is called the $N$-dimensional Vicsek cable system. Each closed (open) cable is a(n) closed (open) interval in $\R^N$ and
$$X=\bigcup_{\mbox{\tiny
$
\begin{subarray}{c}
p,q\in V\\
|p-q|=1
\end{subarray}
$
}}
[p,q]\subseteq\R^N,$$
here $[p,q]$ denotes the closed interval with endpoints $p,q\in\R^N$. It can be easily checked (\cite[Equation (4.14)]{BCG01}) that \ref{eqn_VPhi} holds with $\alpha=\log(2^N+1)/\log3$.

For any $n\ge0$, we say that a subset $W$ of $X$ is an $n$-skeleton if $W$ is a translation of the intersection of the closed convex hull of $V^{(n)}$ and $X$. It is obvious that the closed convex hull of $W$ is a cube, we say that the $2^N$ vertices of the cube are the boundary points of the skeleton and the center of the cube is the center of the skeleton.

Let $n\ge0$ and let $W$ be an $n$-skeleton; denote by $q_1,\ldots,q_{2^N}$ its boundary points and $q_{2^N+1}$ its center, see Figure \ref{fig_n_Vicsek} for $N=2$. Let $u$ be a harmonic function in $W\backslash\{q_1,\ldots,q_{2^N}\}$ with $u(q_i)=a_i$, $i=1,\ldots,2^N$. The fact that $(V,E)$ is a tree implies that each point $x\in W\backslash\cup_{i=1}^{2^N}[q_i,q_{2^N+1}]$ can be joined to $\cup_{i=1}^{2^N}[q_i,q_{2^N+1}]$ by a unique path, let $\gamma(x)\in\cup_{i=1}^{2^N}[q_i,q_{2^N+1}]$ denote the other endpoint of the path. For any $x\in\cup_{i=1}^{2^N}[q_i,q_{2^N+1}]$, let $\gamma(x)=x$. Then the harmonicity of $u$ in $W\backslash\{q_1,\ldots,q_{2^N}\}$ is equivalent to the following.
\begin{itemize}
\item $u(q_{2^N+1})=\frac{1}{2^N}\sum_{i=1}^{2^N}a_i$.
\item For any $i=1,\ldots,2^N$, the function $u$ is linear on the closed interval $[q_i,q_{2^N+1}]$.
\item For any $x\in W$, there holds: $u(x)=u(\gamma(x))$.
\end{itemize}

\begin{figure}[ht]
\centering
\begin{tikzpicture}[scale=0.3]
\draw (-0.7,-0.7) node {$q_1$};
\draw (18.7,-0.7) node {$q_2$};
\draw (18.7,18.7) node {$q_3$};
\draw (-0.7,18.7) node {$q_4$};

\draw (10.3,9) node {$q_5$};

\draw (0,2)--(2,0);
\draw (0,0)--(2,2);
\draw (4,0)--(6,2);
\draw (4,2)--(6,0);
\draw (4,6)--(6,4);
\draw (4,4)--(6,6);
\draw (0,4)--(2,6);
\draw (2,4)--(0,6);
\draw (2,2)--(4,4);
\draw (2,4)--(4,2);

\draw (0+12,2)--(2+12,0);
\draw (0+12,0)--(2+12,2);
\draw (4+12,0)--(6+12,2);
\draw (4+12,2)--(6+12,0);
\draw (4+12,6)--(6+12,4);
\draw (4+12,4)--(6+12,6);
\draw (0+12,4)--(2+12,6);
\draw (2+12,4)--(0+12,6);
\draw (2+12,2)--(4+12,4);
\draw (2+12,4)--(4+12,2);

\draw (0+6,2+6)--(2+6,0+6);
\draw (0+6,0+6)--(2+6,2+6);
\draw (4+6,0+6)--(6+6,2+6);
\draw (4+6,2+6)--(6+6,0+6);
\draw (4+6,6+6)--(6+6,4+6);
\draw (4+6,4+6)--(6+6,6+6);
\draw (0+6,4+6)--(2+6,6+6);
\draw (2+6,4+6)--(0+6,6+6);
\draw (2+6,2+6)--(4+6,4+6);
\draw (2+6,4+6)--(4+6,2+6);

\draw (0,2+12)--(2,0+12);
\draw (0,0+12)--(2,2+12);
\draw (4,0+12)--(6,2+12);
\draw (4,2+12)--(6,0+12);
\draw (4,6+12)--(6,4+12);
\draw (4,4+12)--(6,6+12);
\draw (0,4+12)--(2,6+12);
\draw (2,4+12)--(0,6+12);
\draw (2,2+12)--(4,4+12);
\draw (2,4+12)--(4,2+12);

\draw (0+12,2+12)--(2+12,0+12);
\draw (0+12,0+12)--(2+12,2+12);
\draw (4+12,0+12)--(6+12,2+12);
\draw (4+12,2+12)--(6+12,0+12);
\draw (4+12,6+12)--(6+12,4+12);
\draw (4+12,4+12)--(6+12,6+12);
\draw (0+12,4+12)--(2+12,6+12);
\draw (2+12,4+12)--(0+12,6+12);
\draw (2+12,2+12)--(4+12,4+12);
\draw (2+12,4+12)--(4+12,2+12);

\draw[fill=black] (0,0) circle (0.25);
\draw[fill=black] (0,2) circle (0.25);
\draw[fill=black] (2,2) circle (0.25);
\draw[fill=black] (2,0) circle (0.25);
\draw[fill=black] (1,1) circle (0.25);

\draw[fill=black] (0+4,0) circle (0.25);
\draw[fill=black] (0+4,2) circle (0.25);
\draw[fill=black] (2+4,2) circle (0.25);
\draw[fill=black] (2+4,0) circle (0.25);
\draw[fill=black] (1+4,1) circle (0.25);

\draw[fill=black] (0,0+4) circle (0.25);
\draw[fill=black] (0,2+4) circle (0.25);
\draw[fill=black] (2,2+4) circle (0.25);
\draw[fill=black] (2,0+4) circle (0.25);
\draw[fill=black] (1,1+4) circle (0.25);

\draw[fill=black] (0+4,0+4) circle (0.25);
\draw[fill=black] (0+4,2+4) circle (0.25);
\draw[fill=black] (2+4,2+4) circle (0.25);
\draw[fill=black] (2+4,0+4) circle (0.25);
\draw[fill=black] (1+4,1+4) circle (0.25);

\draw[fill=black] (0+2,0+2) circle (0.25);
\draw[fill=black] (0+2,2+2) circle (0.25);
\draw[fill=black] (2+2,2+2) circle (0.25);
\draw[fill=black] (2+2,0+2) circle (0.25);
\draw[fill=black] (1+2,1+2) circle (0.25);

\draw[fill=black] (0+12,0) circle (0.25);
\draw[fill=black] (0+12,2) circle (0.25);
\draw[fill=black] (2+12,2) circle (0.25);
\draw[fill=black] (2+12,0) circle (0.25);
\draw[fill=black] (1+12,1) circle (0.25);

\draw[fill=black] (0+4+12,0) circle (0.25);
\draw[fill=black] (0+4+12,2) circle (0.25);
\draw[fill=black] (2+4+12,2) circle (0.25);
\draw[fill=black] (2+4+12,0) circle (0.25);
\draw[fill=black] (1+4+12,1) circle (0.25);

\draw[fill=black] (0+12,0+4) circle (0.25);
\draw[fill=black] (0+12,2+4) circle (0.25);
\draw[fill=black] (2+12,2+4) circle (0.25);
\draw[fill=black] (2+12,0+4) circle (0.25);
\draw[fill=black] (1+12,1+4) circle (0.25);

\draw[fill=black] (0+4+12,0+4) circle (0.25);
\draw[fill=black] (0+4+12,2+4) circle (0.25);
\draw[fill=black] (2+4+12,2+4) circle (0.25);
\draw[fill=black] (2+4+12,0+4) circle (0.25);
\draw[fill=black] (1+4+12,1+4) circle (0.25);

\draw[fill=black] (0+2+12,0+2) circle (0.25);
\draw[fill=black] (0+2+12,2+2) circle (0.25);
\draw[fill=black] (2+2+12,2+2) circle (0.25);
\draw[fill=black] (2+2+12,0+2) circle (0.25);
\draw[fill=black] (1+2+12,1+2) circle (0.25);

\draw[fill=black] (0,0+12) circle (0.25);
\draw[fill=black] (0,2+12) circle (0.25);
\draw[fill=black] (2,2+12) circle (0.25);
\draw[fill=black] (2,0+12) circle (0.25);
\draw[fill=black] (1,1+12) circle (0.25);

\draw[fill=black] (0+4,0+12) circle (0.25);
\draw[fill=black] (0+4,2+12) circle (0.25);
\draw[fill=black] (2+4,2+12) circle (0.25);
\draw[fill=black] (2+4,0+12) circle (0.25);
\draw[fill=black] (1+4,1+12) circle (0.25);

\draw[fill=black] (0,0+4+12) circle (0.25);
\draw[fill=black] (0,2+4+12) circle (0.25);
\draw[fill=black] (2,2+4+12) circle (0.25);
\draw[fill=black] (2,0+4+12) circle (0.25);
\draw[fill=black] (1,1+4+12) circle (0.25);

\draw[fill=black] (0+4,0+4+12) circle (0.25);
\draw[fill=black] (0+4,2+4+12) circle (0.25);
\draw[fill=black] (2+4,2+4+12) circle (0.25);
\draw[fill=black] (2+4,0+4+12) circle (0.25);
\draw[fill=black] (1+4,1+4+12) circle (0.25);

\draw[fill=black] (0+2,0+2+12) circle (0.25);
\draw[fill=black] (0+2,2+2+12) circle (0.25);
\draw[fill=black] (2+2,2+2+12) circle (0.25);
\draw[fill=black] (2+2,0+2+12) circle (0.25);
\draw[fill=black] (1+2,1+2+12) circle (0.25);

\draw[fill=black] (0+12,0+12) circle (0.25);
\draw[fill=black] (0+12,2+12) circle (0.25);
\draw[fill=black] (2+12,2+12) circle (0.25);
\draw[fill=black] (2+12,0+12) circle (0.25);
\draw[fill=black] (1+12,1+12) circle (0.25);

\draw[fill=black] (0+4+12,0+12) circle (0.25);
\draw[fill=black] (0+4+12,2+12) circle (0.25);
\draw[fill=black] (2+4+12,2+12) circle (0.25);
\draw[fill=black] (2+4+12,0+12) circle (0.25);
\draw[fill=black] (1+4+12,1+12) circle (0.25);

\draw[fill=black] (0+12,0+4+12) circle (0.25);
\draw[fill=black] (0+12,2+4+12) circle (0.25);
\draw[fill=black] (2+12,2+4+12) circle (0.25);
\draw[fill=black] (2+12,0+4+12) circle (0.25);
\draw[fill=black] (1+12,1+4+12) circle (0.25);

\draw[fill=black] (0+4+12,0+4+12) circle (0.25);
\draw[fill=black] (0+4+12,2+4+12) circle (0.25);
\draw[fill=black] (2+4+12,2+4+12) circle (0.25);
\draw[fill=black] (2+4+12,0+4+12) circle (0.25);
\draw[fill=black] (1+4+12,1+4+12) circle (0.25);

\draw[fill=black] (0+2+12,0+2+12) circle (0.25);
\draw[fill=black] (0+2+12,2+2+12) circle (0.25);
\draw[fill=black] (2+2+12,2+2+12) circle (0.25);
\draw[fill=black] (2+2+12,0+2+12) circle (0.25);
\draw[fill=black] (1+2+12,1+2+12) circle (0.25);

\draw[fill=black] (0+6,0+6) circle (0.25);
\draw[fill=black] (0+6,2+6) circle (0.25);
\draw[fill=black] (2+6,2+6) circle (0.25);
\draw[fill=black] (2+6,0+6) circle (0.25);
\draw[fill=black] (1+6,1+6) circle (0.25);

\draw[fill=black] (0+4+6,0+6) circle (0.25);
\draw[fill=black] (0+4+6,2+6) circle (0.25);
\draw[fill=black] (2+4+6,2+6) circle (0.25);
\draw[fill=black] (2+4+6,0+6) circle (0.25);
\draw[fill=black] (1+4+6,1+6) circle (0.25);

\draw[fill=black] (0+6,0+4+6) circle (0.25);
\draw[fill=black] (0+6,2+4+6) circle (0.25);
\draw[fill=black] (2+6,2+4+6) circle (0.25);
\draw[fill=black] (2+6,0+4+6) circle (0.25);
\draw[fill=black] (1+6,1+4+6) circle (0.25);

\draw[fill=black] (0+4+6,0+4+6) circle (0.25);
\draw[fill=black] (0+4+6,2+4+6) circle (0.25);
\draw[fill=black] (2+4+6,2+4+6) circle (0.25);
\draw[fill=black] (2+4+6,0+4+6) circle (0.25);
\draw[fill=black] (1+4+6,1+4+6) circle (0.25);

\draw[fill=black] (0+2+6,0+2+6) circle (0.25);
\draw[fill=black] (0+2+6,2+2+6) circle (0.25);
\draw[fill=black] (2+2+6,2+2+6) circle (0.25);
\draw[fill=black] (2+2+6,0+2+6) circle (0.25);
\draw[fill=black] (1+2+6,1+2+6) circle (0.25);

\end{tikzpicture}
\caption{An $n$-Skeleton in the 2-Dimensional Vicsek Cable System}\label{fig_n_Vicsek}
\end{figure}
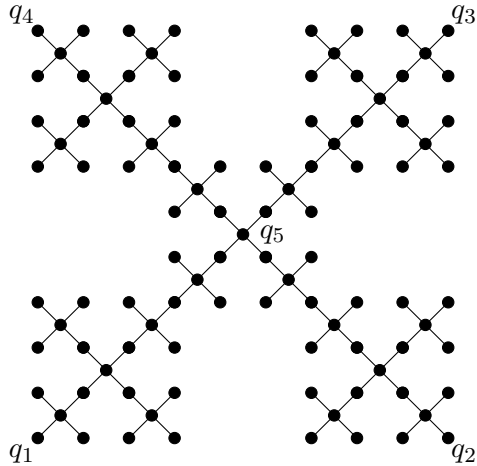

We now describe the Sierpi\'nski cable system. In $\R^2$, let $p_1=(0,0)$, $p_2=(1,0)$ and $p_3=(\frac{1}{2},\frac{\sqrt{3}}{2})$. Let $f_i(x)=\frac{1}{2}(x+p_i)$, $x\in\R^2$, $i=1,2,3$. Then the Sierpi\'nski gasket is the unique non-empty compact set $K$ in $\R^2$ satisfying $K=\cup_{i=1}^3f_i(K)$.

Let $V_0=\myset{p_1,p_2,p_3}$ and $V_{n+1}=\cup_{i=1}^3f_i(V_n)$ for any $n\ge0$. Then $\myset{V_n}_{n\ge0}$ is an increasing sequence of finite subsets of $K$ and the closure of $\cup_{n\ge0}V_n$ is $K$.

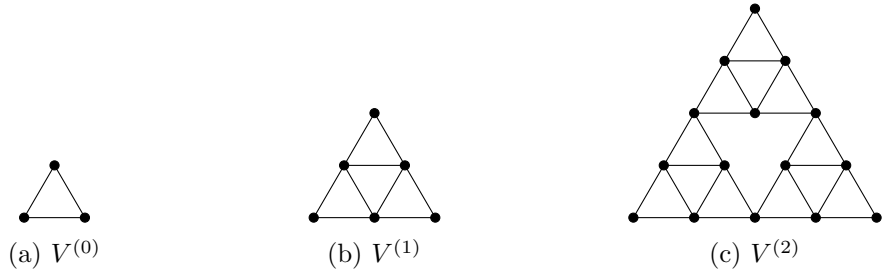
\begin{figure}[ht]
\centering
\begin{subfigure}[b]{0.2\textwidth}
\centering
\begin{tikzpicture}[scale=0.2]
\draw (0,0)--(4,0)--(2,2*1.73205080756887)--cycle;

\draw[fill=black] (0,0) circle (0.3);
\draw[fill=black] (4,0) circle (0.3);
\draw[fill=black] (2,2*1.73205080756887) circle (0.3);
\end{tikzpicture}
\caption{$V^{(0)}$}
\end{subfigure}
\hspace{2em}
\begin{subfigure}[b]{0.2\textwidth}
\centering
\begin{tikzpicture}[scale=0.4]
\draw (0,0)--(4,0)--(2,2*1.73205080756887)--cycle;
\draw (2,0)--(3,1.73205080756887)--(1,1.73205080756887)--cycle;

\draw[fill=black] (1,1*1.73205080756887) circle (0.3/2);
\draw[fill=black] (3,1*1.73205080756887) circle (0.3/2);
\draw[fill=black] (2,2*1.73205080756887) circle (0.3/2);
\draw[fill=black] (0,0) circle (0.3/2);
\draw[fill=black] (2,0) circle (0.3/2);
\draw[fill=black] (4,0) circle (0.3/2);
\end{tikzpicture}
\caption{$V^{(1)}$}
\end{subfigure}
\hspace{2em}
\begin{subfigure}[b]{0.3\textwidth}
\centering
\begin{tikzpicture}[scale=0.8]
\draw (0,0)--(4,0)--(2,2*1.73205080756887)--cycle;
\draw (2,0)--(3,1.73205080756887)--(1,1.73205080756887)--cycle;
\draw (1,0)--(1.5,0.5*1.73205080756887)--(0.5,0.5*1.73205080756887)--cycle;
\draw (3,0)--(3.5,0.5*1.73205080756887)--(2.5,0.5*1.73205080756887)--cycle;
\draw (2,1.73205080756887)--(1.5,1.5*1.73205080756887)--(2.5,1.5*1.73205080756887)--cycle;

\draw[fill=black] (0.5,0.5*1.73205080756887) circle (0.3/4);
\draw[fill=black] (1,1*1.73205080756887) circle (0.3/4);
\draw[fill=black] (1.5,1.5*1.73205080756887) circle (0.3/4);
\draw[fill=black] (2,2*1.73205080756887) circle (0.3/4);
\draw[fill=black] (2.5,1.5*1.73205080756887) circle (0.3/4);
\draw[fill=black] (3,1*1.73205080756887) circle (0.3/4);
\draw[fill=black] (3.5,0.5*1.73205080756887) circle (0.3/4);

\draw[fill=black] (0,0) circle (0.3/4);
\draw[fill=black] (1,0) circle (0.3/4);
\draw[fill=black] (2,0) circle (0.3/4);
\draw[fill=black] (3,0) circle (0.3/4);
\draw[fill=black] (4,0) circle (0.3/4);

\draw[fill=black] (1.5,0.5*1.73205080756887) circle (0.3/4);
\draw[fill=black] (2.5,0.5*1.73205080756887) circle (0.3/4);
\draw[fill=black] (2,1*1.73205080756887) circle (0.3/4);

\end{tikzpicture}
\caption{$V^{(2)}$}
\end{subfigure}
\caption{$V^{(0)}$, $V^{(1)}$ and $V^{(2)}$}\label{fig_V012_Sierpinski}
\end{figure}

For any $n\ge0$, let $V^{(n)}=2^nV_n=\{2^nv:v\in V_n\}$, see Figure \ref{fig_V012_Sierpinski} for $V^{(0)}$, $V^{(1)}$ and $V^{(2)}$. Then $\myset{V^{(n)}}_{n\ge0}$ is an increasing sequence of finite sets. Let $V=\cup_{n\ge0}V^{(n)}$ and $E=\{\{p,q\}:p,q\in V,|p-q|=1\}$, then $(V,E)$ is an infinite, locally bounded, connected graph, the corresponding unbounded cable system is called the Sierpi\'nski cable system. Each closed (open) cable is a(n) closed (open) interval in $\R^2$ and
$$X=\bigcup_{\mbox{\tiny
$
\begin{subarray}{c}
p,q\in V\\
|p-q|=1
\end{subarray}
$
}}
[p,q]\subseteq\R^2,$$
here $[p,q]$ denotes the closed interval with endpoints $p,q\in\R^2$. It is well-known (\cite[Lemma 2.1]{B98}) that \ref{eqn_VPhi} holds with $\alpha=\log3/\log2$.

For any $n\ge0$, we say that a subset $W$ of $X$ is an $n$-skeleton if $W$ is a translation of the intersection of the closed convex hull of $V^{(n)}$ and $X$. It is obvious that the closed convex hull of $W$ is an equilateral triangle, we call the three vertices of the triangle the boundary points of the skeleton.

Let $n\ge1$ and let $W$ be an $n$-skeleton with boundary points $q_1,q_2,q_3$. Let $q_4,q_5,q_6\in V$ denote the midpoints of the closed intervals $[q_1,q_2]$, $[q_2,q_3]$, $[q_3,q_1]$, respectively, see Figure \ref{fig_n_Sierpinski}. Let $u$ be a harmonic function in $W\backslash\{q_1,q_2,q_3\}$ with $u(q_i)=a_i$, $i=1,2,3$. By the standard $\frac{2}{5}$-$\frac{2}{5}$-$\frac{1}{5}$-algorithm (see \cite[Proposition 3.2.1, Example 3.2.6]{Kig01}), we have
\begin{align}
u(q_4)&=\frac{2}{5}a_1+\frac{2}{5}a_2+\frac{1}{5}a_3,\nonumber\\
u(q_5)&=\frac{1}{5}a_1+\frac{2}{5}a_2+\frac{2}{5}a_3,\label{eqn_2215}\\
u(q_6)&=\frac{2}{5}a_1+\frac{1}{5}a_2+\frac{2}{5}a_3.\nonumber
\end{align}

\begin{figure}[ht]
\centering
\begin{tikzpicture}[scale=1.2]
\draw (-0.3,-0.3) node {$q_1$};
\draw (4.3,-0.3) node {$q_2$};
\draw (2,2*1.73205080756887+0.3) node {$q_3$};

\draw (2,-0.3) node {$q_4$};
\draw (3+0.3,1*1.73205080756887) node {$q_5$};
\draw (1-0.3,1*1.73205080756887) node {$q_6$};

\draw (0,0)--(4,0)--(2,2*1.73205080756887)--cycle;
\draw (2,0)--(3,1.73205080756887)--(1,1.73205080756887)--cycle;
\draw (1,0)--(1.5,0.5*1.73205080756887)--(0.5,0.5*1.73205080756887)--cycle;
\draw (3,0)--(3.5,0.5*1.73205080756887)--(2.5,0.5*1.73205080756887)--cycle;
\draw (2,1.73205080756887)--(1.5,1.5*1.73205080756887)--(2.5,1.5*1.73205080756887)--cycle;

\draw[fill=black] (0.5,0.5*1.73205080756887) circle (0.3/4);
\draw[fill=black] (1,1*1.73205080756887) circle (0.3/4);
\draw[fill=black] (1.5,1.5*1.73205080756887) circle (0.3/4);
\draw[fill=black] (2,2*1.73205080756887) circle (0.3/4);
\draw[fill=black] (2.5,1.5*1.73205080756887) circle (0.3/4);
\draw[fill=black] (3,1*1.73205080756887) circle (0.3/4);
\draw[fill=black] (3.5,0.5*1.73205080756887) circle (0.3/4);

\draw[fill=black] (0,0) circle (0.3/4);
\draw[fill=black] (1,0) circle (0.3/4);
\draw[fill=black] (2,0) circle (0.3/4);
\draw[fill=black] (3,0) circle (0.3/4);
\draw[fill=black] (4,0) circle (0.3/4);

\draw[fill=black] (1.5,0.5*1.73205080756887) circle (0.3/4);
\draw[fill=black] (2.5,0.5*1.73205080756887) circle (0.3/4);
\draw[fill=black] (2,1*1.73205080756887) circle (0.3/4);

\end{tikzpicture}
\caption{An $n$-Skeleton in the Sierpi\'nski Cable System}\label{fig_n_Sierpinski}
\end{figure}
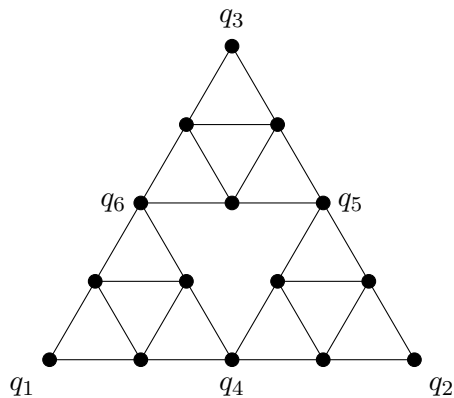

It can be shown that \hyperlink{eqn_HKPsi}{HK($\Psi$)} holds with $\beta=\log(3\cdot(2^N+1))/\log3$ for the $N$-dimensional Vicsek cable system and $\beta=\log5/\log2$ for the Sierpi\'nski cable system. For example, it is easy to check that the conditions $(H)$ and $(R_F)$ from \cite{GH14a} hold with $F=\Psi$, then it follows from \cite[Theorem 3.14]{GH14a} that the conditions $(UE)$ and $(NLE)$ from \cite{GH14a} hold, that is, \ref{eqn_UHK} and \ref{eqn_NLE} hold. Since $(X,d)$ is geodesic, we have \hyperlink{eqn_HKPsi}{HK($\Psi$)}. By Proposition \ref{prop_UHK}, we have \ref{eqn_FK} and \ref{eqn_CS}, so that the results about Poisson equation in Subsection \ref{subsec_Poi} apply.

\section{Reverse H\"older Inequalities}\label{sec_GRH}

To start with, we show that \ref{eqn_RH} holds on the $N$-dimensional Vicsek cable system.

\begin{myprop}\label{prop_RH_Vicsek}
The reverse H\"older  inequality \ref{eqn_RH} holds on the $N$-dimensional Vicsek cable system.
\end{myprop}

\begin{proof}
Let $B$ be a ball with radius $r$. If $r\in(0,27)$, then the conclusion follows from the result on intervals.  Therefore, we may assume that $r\in[27,+\infty)$.

For any $x\in B\backslash V$, there exist $p,q\in V\cap 2B$ with $|p-q|=1$ such that $x\in(p,q)\subseteq2B$. Since $u$ is harmonic in $2B$, we have $|\nabla u(x)|=|u(p)-u(q)|$. Let $n\ge 2$ be the positive integer satisfying $3^{n+1}\le r<3^{n+2}$, then there exists an $n$-skeleton $W$ satisfying $p,q\in W\subseteq 2B$. Therefore,
$$m(W)=2^N\cdot(2^N+1)^{n}\le m(2B)\le C_R\Phi(2r)\le C_R(3\cdot3^{n+2})^\alpha=C_R(2^N+1)^{n+3},$$
where $C_R$ is the constant in \ref{eqn_VPhi}.

Let $q_1,\ldots,q_{2^N}$ be the boundary points of $W$ and $q_{2^{N}+1}$ the center of $W$. Then we have
\begin{align*}
|u(p)-u(q)|&\le\frac{1}{3^{n}}\max\{|u(q_i)-u(q_{2^N+1})|:i=1,\ldots,2^N\}\\
&=\frac{1}{3^{n}}\max\{|u(q_i)-\frac{1}{2^N}\sum_{i=1}^{2^N}u(q_i)|:i=1,\ldots,2^N\}\\
&\le\frac{2}{3^n}\max\{|u(q_i)|:i=1,\ldots,2^N\}.
\end{align*}

Without loss of generality, we may assume that $u(q_1)>0$ and $|u(q_1)|=\max\{|u(q_i)|:i=1,\ldots,2^N\}$. Let $W_0$ be the $(n-2)$-skeleton with a boundary point $q_1$ satisfying $W_0\subseteq W$. Let $q_0$ be the boundary point of $W_0$ that lies in $(q_1,q_{2^N+1})$, see Figure \ref{fig_look_Vicsek} for $N=2$.

\begin{figure}[ht]
\centering
\begin{tikzpicture}[scale=0.3]
\draw (-0.7,-0.7) node {$q_1$};
\draw (18.7,-0.7) node {$q_2$};
\draw (18.7,18.7) node {$q_3$};
\draw (-0.7,18.7) node {$q_4$};

\draw (10.3,9) node {$q_5$};

\draw (2.7,1.3) node {$q_0$};

\draw (0,2)--(2,0);
\draw (0,0)--(2,2);
\draw (4,0)--(6,2);
\draw (4,2)--(6,0);
\draw (4,6)--(6,4);
\draw (4,4)--(6,6);
\draw (0,4)--(2,6);
\draw (2,4)--(0,6);
\draw (2,2)--(4,4);
\draw (2,4)--(4,2);

\draw (0+12,2)--(2+12,0);
\draw (0+12,0)--(2+12,2);
\draw (4+12,0)--(6+12,2);
\draw (4+12,2)--(6+12,0);
\draw (4+12,6)--(6+12,4);
\draw (4+12,4)--(6+12,6);
\draw (0+12,4)--(2+12,6);
\draw (2+12,4)--(0+12,6);
\draw (2+12,2)--(4+12,4);
\draw (2+12,4)--(4+12,2);

\draw (0+6,2+6)--(2+6,0+6);
\draw (0+6,0+6)--(2+6,2+6);
\draw (4+6,0+6)--(6+6,2+6);
\draw (4+6,2+6)--(6+6,0+6);
\draw (4+6,6+6)--(6+6,4+6);
\draw (4+6,4+6)--(6+6,6+6);
\draw (0+6,4+6)--(2+6,6+6);
\draw (2+6,4+6)--(0+6,6+6);
\draw (2+6,2+6)--(4+6,4+6);
\draw (2+6,4+6)--(4+6,2+6);

\draw (0,2+12)--(2,0+12);
\draw (0,0+12)--(2,2+12);
\draw (4,0+12)--(6,2+12);
\draw (4,2+12)--(6,0+12);
\draw (4,6+12)--(6,4+12);
\draw (4,4+12)--(6,6+12);
\draw (0,4+12)--(2,6+12);
\draw (2,4+12)--(0,6+12);
\draw (2,2+12)--(4,4+12);
\draw (2,4+12)--(4,2+12);

\draw (0+12,2+12)--(2+12,0+12);
\draw (0+12,0+12)--(2+12,2+12);
\draw (4+12,0+12)--(6+12,2+12);
\draw (4+12,2+12)--(6+12,0+12);
\draw (4+12,6+12)--(6+12,4+12);
\draw (4+12,4+12)--(6+12,6+12);
\draw (0+12,4+12)--(2+12,6+12);
\draw (2+12,4+12)--(0+12,6+12);
\draw (2+12,2+12)--(4+12,4+12);
\draw (2+12,4+12)--(4+12,2+12);

\draw[fill=black] (0,0) circle (0.25);
\draw[fill=black] (0,2) circle (0.25);
\draw[fill=black] (2,2) circle (0.25);
\draw[fill=black] (2,0) circle (0.25);
\draw[fill=black] (1,1) circle (0.25);

\draw[fill=black] (0+4,0) circle (0.25);
\draw[fill=black] (0+4,2) circle (0.25);
\draw[fill=black] (2+4,2) circle (0.25);
\draw[fill=black] (2+4,0) circle (0.25);
\draw[fill=black] (1+4,1) circle (0.25);

\draw[fill=black] (0,0+4) circle (0.25);
\draw[fill=black] (0,2+4) circle (0.25);
\draw[fill=black] (2,2+4) circle (0.25);
\draw[fill=black] (2,0+4) circle (0.25);
\draw[fill=black] (1,1+4) circle (0.25);

\draw[fill=black] (0+4,0+4) circle (0.25);
\draw[fill=black] (0+4,2+4) circle (0.25);
\draw[fill=black] (2+4,2+4) circle (0.25);
\draw[fill=black] (2+4,0+4) circle (0.25);
\draw[fill=black] (1+4,1+4) circle (0.25);

\draw[fill=black] (0+2,0+2) circle (0.25);
\draw[fill=black] (0+2,2+2) circle (0.25);
\draw[fill=black] (2+2,2+2) circle (0.25);
\draw[fill=black] (2+2,0+2) circle (0.25);
\draw[fill=black] (1+2,1+2) circle (0.25);

\draw[fill=black] (0+12,0) circle (0.25);
\draw[fill=black] (0+12,2) circle (0.25);
\draw[fill=black] (2+12,2) circle (0.25);
\draw[fill=black] (2+12,0) circle (0.25);
\draw[fill=black] (1+12,1) circle (0.25);

\draw[fill=black] (0+4+12,0) circle (0.25);
\draw[fill=black] (0+4+12,2) circle (0.25);
\draw[fill=black] (2+4+12,2) circle (0.25);
\draw[fill=black] (2+4+12,0) circle (0.25);
\draw[fill=black] (1+4+12,1) circle (0.25);

\draw[fill=black] (0+12,0+4) circle (0.25);
\draw[fill=black] (0+12,2+4) circle (0.25);
\draw[fill=black] (2+12,2+4) circle (0.25);
\draw[fill=black] (2+12,0+4) circle (0.25);
\draw[fill=black] (1+12,1+4) circle (0.25);

\draw[fill=black] (0+4+12,0+4) circle (0.25);
\draw[fill=black] (0+4+12,2+4) circle (0.25);
\draw[fill=black] (2+4+12,2+4) circle (0.25);
\draw[fill=black] (2+4+12,0+4) circle (0.25);
\draw[fill=black] (1+4+12,1+4) circle (0.25);

\draw[fill=black] (0+2+12,0+2) circle (0.25);
\draw[fill=black] (0+2+12,2+2) circle (0.25);
\draw[fill=black] (2+2+12,2+2) circle (0.25);
\draw[fill=black] (2+2+12,0+2) circle (0.25);
\draw[fill=black] (1+2+12,1+2) circle (0.25);

\draw[fill=black] (0,0+12) circle (0.25);
\draw[fill=black] (0,2+12) circle (0.25);
\draw[fill=black] (2,2+12) circle (0.25);
\draw[fill=black] (2,0+12) circle (0.25);
\draw[fill=black] (1,1+12) circle (0.25);

\draw[fill=black] (0+4,0+12) circle (0.25);
\draw[fill=black] (0+4,2+12) circle (0.25);
\draw[fill=black] (2+4,2+12) circle (0.25);
\draw[fill=black] (2+4,0+12) circle (0.25);
\draw[fill=black] (1+4,1+12) circle (0.25);

\draw[fill=black] (0,0+4+12) circle (0.25);
\draw[fill=black] (0,2+4+12) circle (0.25);
\draw[fill=black] (2,2+4+12) circle (0.25);
\draw[fill=black] (2,0+4+12) circle (0.25);
\draw[fill=black] (1,1+4+12) circle (0.25);

\draw[fill=black] (0+4,0+4+12) circle (0.25);
\draw[fill=black] (0+4,2+4+12) circle (0.25);
\draw[fill=black] (2+4,2+4+12) circle (0.25);
\draw[fill=black] (2+4,0+4+12) circle (0.25);
\draw[fill=black] (1+4,1+4+12) circle (0.25);

\draw[fill=black] (0+2,0+2+12) circle (0.25);
\draw[fill=black] (0+2,2+2+12) circle (0.25);
\draw[fill=black] (2+2,2+2+12) circle (0.25);
\draw[fill=black] (2+2,0+2+12) circle (0.25);
\draw[fill=black] (1+2,1+2+12) circle (0.25);

\draw[fill=black] (0+12,0+12) circle (0.25);
\draw[fill=black] (0+12,2+12) circle (0.25);
\draw[fill=black] (2+12,2+12) circle (0.25);
\draw[fill=black] (2+12,0+12) circle (0.25);
\draw[fill=black] (1+12,1+12) circle (0.25);

\draw[fill=black] (0+4+12,0+12) circle (0.25);
\draw[fill=black] (0+4+12,2+12) circle (0.25);
\draw[fill=black] (2+4+12,2+12) circle (0.25);
\draw[fill=black] (2+4+12,0+12) circle (0.25);
\draw[fill=black] (1+4+12,1+12) circle (0.25);

\draw[fill=black] (0+12,0+4+12) circle (0.25);
\draw[fill=black] (0+12,2+4+12) circle (0.25);
\draw[fill=black] (2+12,2+4+12) circle (0.25);
\draw[fill=black] (2+12,0+4+12) circle (0.25);
\draw[fill=black] (1+12,1+4+12) circle (0.25);

\draw[fill=black] (0+4+12,0+4+12) circle (0.25);
\draw[fill=black] (0+4+12,2+4+12) circle (0.25);
\draw[fill=black] (2+4+12,2+4+12) circle (0.25);
\draw[fill=black] (2+4+12,0+4+12) circle (0.25);
\draw[fill=black] (1+4+12,1+4+12) circle (0.25);

\draw[fill=black] (0+2+12,0+2+12) circle (0.25);
\draw[fill=black] (0+2+12,2+2+12) circle (0.25);
\draw[fill=black] (2+2+12,2+2+12) circle (0.25);
\draw[fill=black] (2+2+12,0+2+12) circle (0.25);
\draw[fill=black] (1+2+12,1+2+12) circle (0.25);

\draw[fill=black] (0+6,0+6) circle (0.25);
\draw[fill=black] (0+6,2+6) circle (0.25);
\draw[fill=black] (2+6,2+6) circle (0.25);
\draw[fill=black] (2+6,0+6) circle (0.25);
\draw[fill=black] (1+6,1+6) circle (0.25);

\draw[fill=black] (0+4+6,0+6) circle (0.25);
\draw[fill=black] (0+4+6,2+6) circle (0.25);
\draw[fill=black] (2+4+6,2+6) circle (0.25);
\draw[fill=black] (2+4+6,0+6) circle (0.25);
\draw[fill=black] (1+4+6,1+6) circle (0.25);

\draw[fill=black] (0+6,0+4+6) circle (0.25);
\draw[fill=black] (0+6,2+4+6) circle (0.25);
\draw[fill=black] (2+6,2+4+6) circle (0.25);
\draw[fill=black] (2+6,0+4+6) circle (0.25);
\draw[fill=black] (1+6,1+4+6) circle (0.25);

\draw[fill=black] (0+4+6,0+4+6) circle (0.25);
\draw[fill=black] (0+4+6,2+4+6) circle (0.25);
\draw[fill=black] (2+4+6,2+4+6) circle (0.25);
\draw[fill=black] (2+4+6,0+4+6) circle (0.25);
\draw[fill=black] (1+4+6,1+4+6) circle (0.25);

\draw[fill=black] (0+2+6,0+2+6) circle (0.25);
\draw[fill=black] (0+2+6,2+2+6) circle (0.25);
\draw[fill=black] (2+2+6,2+2+6) circle (0.25);
\draw[fill=black] (2+2+6,0+2+6) circle (0.25);
\draw[fill=black] (1+2+6,1+2+6) circle (0.25);

\end{tikzpicture}
\caption{Looking at an $(n-2)$-Skeleton in the $2$-Dimensional Vicsek Cable System}\label{fig_look_Vicsek}
\end{figure}
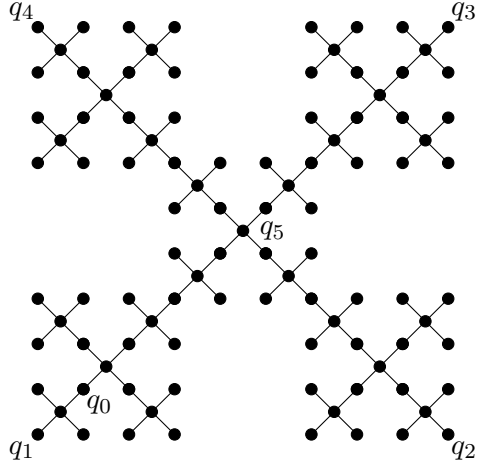

By the maximum principle, we have $u(q_{2^{N}+1})\ge-u(q_1)$, hence
$$u(q_0)=\frac{7}{9}u(q_1)+\frac{2}{9}u(q_{2^N+1})\ge\frac{5}{9}u(q_1).$$
Note that $u$ is harmonic in the open set $W_0\setminus \left\{q_0,q_1\right\}$, and moreover $u(q_0)\ge \frac 59u(q_1)$, $u(q_1)\ge \frac 59 u(q_1)$. By the maximum principle again, we have
$$u\ge\frac{5}{9}u(q_1)>0\text{ on }W_0.$$
Hence
\begin{align*}
&|\nabla u(x)|=|u(p)-u(q)|\le\frac{2}{3^n}u(q_1)\le\frac{2}{3^n}\frac{9}{5}\dashint_{W_0}u\md m\\
&\le\frac{2}{5}\frac{1}{3^{n-2}}\frac{1}{2^N\cdot(2^N+1)^{n-2}}\int_{2B}|u|\md m\le\frac{3^4\cdot(2^N+1)^5C_R}{5\cdot2^{N-1}}\frac{1}{3^{n+2}}\dashint_{2B}|u|\md m\le\frac{C}{r}\dashint_{2B}|u|\md m,
\end{align*}
where $C=\frac{3^4\cdot(2^N+1)^5C_R}{5\cdot2^{N-1}}$, hence
$$\mynorm{|\nabla u|}_{L^\infty(B)}\le\frac{C}{r}\dashint_{2B}|u|\md m.$$
\end{proof}

We now show that \ref{eqn_RH} does not hold on the Sierpi\'nski cable system as follows.

\begin{myprop}\label{prop_RH}
The reverse H\"older inequality \ref{eqn_RH} does not hold on the Sierpi\'nski cable system.
\end{myprop}

\begin{proof}
Suppose by contradiction that \ref{eqn_RH} holds. For any $n\ge0$, consider the ball $B=B(2^{n+1}p_2,2^{n})$, and let $u\in\calF$ be a harmonic function in $2B=B(2^{n+1}p_2,2^{n+1})$ with $u(p_1)=u(2^{n+1}p_3)=-1$ and $u(2^{n+2}p_2)=u(2^{n+1}p_2+2^{n+1}p_3)=1$, see Figure \ref{fig_RH}. The function $u$ can be obtained by applying the standard $\frac{2}{5}$-$\frac{2}{5}$-$\frac{1}{5}$-algorithm in $2B$, and then extending the function arbitrarily outside $\overline{2B}$ only to ensure that $u\in\calF$. Note that $p_1$, $2^{n+1}p_3$, $2^{n+2}p_2$, $2^{n+1}p_2+2^{n+1}p_3\not\in2B$. It is obvious that $u(2^{n+1}p_2)=0$ and by the maximum principle, $|u|\leq 1$ in $2B$, therefore:
$$\dashint_{2B}|u|\md m\le1.$$

\begin{figure}[ht]
\centering
\begin{tikzpicture}[scale=0.6]
\draw (0,0)--(-4,0)--(-2,2*1.73205080756887)--cycle;
\draw (-2,0)--(-3,1.73205080756887)--(-1,1.73205080756887)--cycle;
\draw (0,0)--(4,0)--(2,2*1.73205080756887)--cycle;
\draw (2,0)--(3,1.73205080756887)--(1,1.73205080756887)--cycle;
\draw (0,-0.5) node {\footnotesize{$2^{n+1}p_2$}};
\draw (-4,-0.5) node {\footnotesize{$p_1$}};
\draw (-2,2*1.73205080756887+0.5) node {\footnotesize{$2^{n+1}p_3$}};
\draw (4,-0.5) node {\footnotesize{$2^{n+2}p_2$}};
\draw (2,2*1.73205080756887+0.5) node {\footnotesize{$2^{n+1}p_2+2^{n+1}p_3$}};

\draw[fill=black] (0,0) circle (0.12);
\draw[fill=black] (4,0) circle (0.12);
\draw[fill=black] (2,2*1.73205080756887) circle (0.12);
\draw[fill=black] (-4,0) circle (0.12);
\draw[fill=black] (-2,2*1.73205080756887) circle (0.12);
\draw[fill=black] (2,0) circle (0.12);
\draw[fill=black] (1,1*1.73205080756887) circle (0.12);
\draw[fill=black] (3,1*1.73205080756887) circle (0.12);
\draw[fill=black] (-2,0) circle (0.12);
\draw[fill=black] (-1,1*1.73205080756887) circle (0.12);
\draw[fill=black] (-3,1*1.73205080756887) circle (0.12);
\end{tikzpicture}
\caption{The Ball $2B=B(2^{n+1}p_2,2^{n+1})$}\label{fig_RH}
\end{figure}
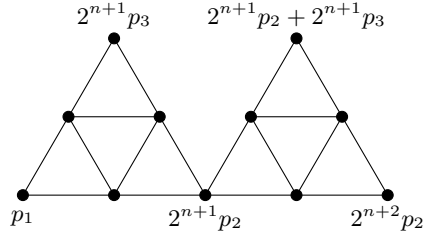

By induction and the standard $\frac{2}{5}$-$\frac{2}{5}$-$\frac{1}{5}$-algorithm Equation (\ref{eqn_2215}), we have
$$u(2^{n+1}p_2+p_2)=u(2^{n+1}p_2+p_3)=\left(\frac{3}{5}\right)^{n+1}.$$
Hence
$$|\nabla u|=\left(\frac{3}{5}\right)^{n+1}\text{ in }(2^{n+1}p_2,2^{n+1}p_2+p_2)\cup(2^{n+1}p_2,2^{n+1}p_2+p_3)\subseteq B.$$
By \ref{eqn_RH}, we have
$$\left(\frac{3}{5}\right)^{n+1}\le\mynorm{|\nabla u|}_{L^\infty(B)}\le\frac{C_H}{2^n}\dashint_{2B}|u|\md m\le\frac{C_H}{2^n},$$
consequently,
$$\left(\frac{6}{5}\right)^{n+1}\le2C_H\text{ for any }n\ge0,$$
contradiction! Hence \ref{eqn_RH} does not hold.
\end{proof}

Proposition \ref{prop_RH} justifies the introduction of the generalized reverse H\"older inequality \ref{eqn_GRH}, which we now show holds both on the $N$-dimensional Vicsek cable system and the Sierpi\'nski cable system.

\begin{myprop}\label{prop_GRH}
The generalized reverse H\"older inequality \ref{eqn_GRH} holds on the $N$-dimensional Vicsek cable system and the Sierpi\'nski cable system.
\end{myprop}

\begin{myrmk}
{\em 
In the small scale, \ref{eqn_GRH} behaves as \ref{eqn_RH}. However, in the large scale, the fractal property comes into effect.
}
\end{myrmk}

\begin{proof}
For the $N$-dimensional Vicsek cable system, since $\frac{\Psi(r)}{\Phi(r)}=r$ for any $r\in(0,+\infty)$, \ref{eqn_GRH} reduces to \ref{eqn_RH}, so that the result follows from Proposition \ref{prop_RH_Vicsek}. Therefore, we only need to consider the Sierpi\'nski cable system. Let $B$ be a ball with radius $r$. If $r\in(0,4)$, then the result follows from the result on intervals. We may thus assume that $r\in[4,+\infty)$.

For any $x\in B\backslash V$, there exist $p,q\in V\cap 2B$ with $|p-q|=1$ such that $x\in(p,q)\subseteq 2B$. Since $u$ is harmonic in $2B$, we have $|\nabla u(x)|=|u(p)-u(q)|$. Let $n\ge2$ be the positive integer satisfying $2^n\le r<2^{n+1}$, then there exists an $n$-skeleton $W$ satisfying $p,q\in W\subseteq 2B$ and

$$m(W)=3^{n+1}\le m(2B)\le C_R\Phi(2r)\le C_R(2\cdot2^{n+1})^\alpha=C_R3^{n+2},$$
where $C_R$ is the constant in \ref{eqn_VPhi}.

Let $q_1,q_2,q_3$ be the boundary points of $W$, see Figure \ref{fig_look_Sierpinski}. Let $F:\R^2\to\R^2$ be the affine mapping that maps $p_i$ to $q_i$, $i=1,2,3$. Let $v$ be the harmonic function on the Sierpi\'nski gasket $K$ with $v(p_i)=u(q_i)$, $i=1,2,3$ (see \cite[Proposition 3.2.1, Example 3.2.6]{Kig01}). Noting that $W\cap V=F(V_n)$, we have $v=u\circ F$ on $V_n$ or $u=v\circ F^{-1}$ on $W\cap V$. Let $i_1,\ldots,i_n\in\{1,2,3\}$ satisfy $F^{-1}(p),F^{-1}(q)\in f_{i_1}\circ\ldots\circ f_{i_n}(K)$.

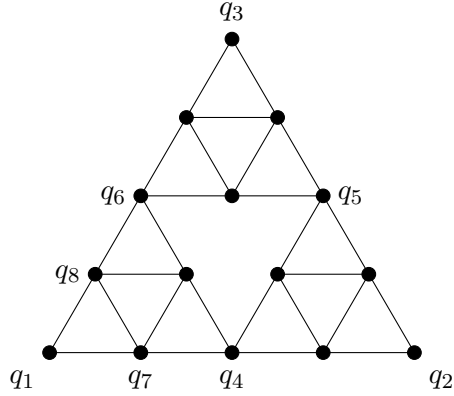
\begin{figure}[ht]
\centering
\begin{tikzpicture}[scale=1.2]
\draw (-0.3,-0.3) node {$q_1$};
\draw (4.3,-0.3) node {$q_2$};
\draw (2,2*1.73205080756887+0.3) node {$q_3$};

\draw (2,-0.3) node {$q_4$};
\draw (3+0.3,1*1.73205080756887) node {$q_5$};
\draw (1-0.3,1*1.73205080756887) node {$q_6$};

\draw (1,-0.3) node {$q_7$};
\draw (0.5-0.3,0.5*1.73205080756887) node {$q_8$};

\draw (0,0)--(4,0)--(2,2*1.73205080756887)--cycle;
\draw (2,0)--(3,1.73205080756887)--(1,1.73205080756887)--cycle;
\draw (1,0)--(1.5,0.5*1.73205080756887)--(0.5,0.5*1.73205080756887)--cycle;
\draw (3,0)--(3.5,0.5*1.73205080756887)--(2.5,0.5*1.73205080756887)--cycle;
\draw (2,1.73205080756887)--(1.5,1.5*1.73205080756887)--(2.5,1.5*1.73205080756887)--cycle;

\draw[fill=black] (0.5,0.5*1.73205080756887) circle (0.3/4);
\draw[fill=black] (1,1*1.73205080756887) circle (0.3/4);
\draw[fill=black] (1.5,1.5*1.73205080756887) circle (0.3/4);
\draw[fill=black] (2,2*1.73205080756887) circle (0.3/4);
\draw[fill=black] (2.5,1.5*1.73205080756887) circle (0.3/4);
\draw[fill=black] (3,1*1.73205080756887) circle (0.3/4);
\draw[fill=black] (3.5,0.5*1.73205080756887) circle (0.3/4);

\draw[fill=black] (0,0) circle (0.3/4);
\draw[fill=black] (1,0) circle (0.3/4);
\draw[fill=black] (2,0) circle (0.3/4);
\draw[fill=black] (3,0) circle (0.3/4);
\draw[fill=black] (4,0) circle (0.3/4);

\draw[fill=black] (1.5,0.5*1.73205080756887) circle (0.3/4);
\draw[fill=black] (2.5,0.5*1.73205080756887) circle (0.3/4);
\draw[fill=black] (2,1*1.73205080756887) circle (0.3/4);

\end{tikzpicture}
\caption{Looking at an $(n-2)$-Skeleton in the Sierpi\'nski Cable System}\label{fig_look_Sierpinski}
\end{figure}

Recall that the oscillation $\mathrm{Osc}(u,D)$ of a function $u$ on a set $D$ is defined as
$$\mathrm{Osc}(u,D):=\sup_Du-\inf_Du.$$

By \cite[THEOREM 8.3]{Str00} or \cite[Theorem 1.3, Example 5.1]{THP14} about H\"older estimates of harmonic functions on the Sierpi\'nski gasket, we have

\begin{align*}
&|u(p)-u(q)|=|v(F^{-1}(p))-v(F^{-1}(q))|\le\mathrm{Osc}(v,f_{i_1}\circ\ldots\circ f_{i_n}(K))\le\left(\frac{3}{5}\right)^{n}\mathrm{Osc}(v,K)\\
&=\left(\frac{3}{5}\right)^n\max\myset{|v(p_i)-v(p_j)|:i,j=1,2,3}=\left(\frac{3}{5}\right)^n\max\myset{|u(q_i)-u(q_j)|:i,j=1,2,3}.
\end{align*}

Without loss of generality, we may assume that $u(q_1)>0$ and $|u(q_1)|=\max_{i=1,2,3}|u(q_i)|$. Let $W_0$ be the $(n-2)$-skeleton with a boundary point $q_1$ satisfying $W_0\subseteq W$. Let $q_7,q_8$ denote the other two boundary points of $W_0$, see Figure \ref{fig_look_Sierpinski}. By the standard $\frac{2}{5}$-$\frac{2}{5}$-$\frac{1}{5}$-algorithm Equation (\ref{eqn_2215}), we have
\begin{align*}
u(q_7)&=\frac{2}{5}u(q_1)+\frac{2}{5}u(q_4)+\frac{1}{5}u(q_6)\\
&=\frac{2}{5}u(q_1)+\frac{2}{5}\left(\frac{2}{5}u(q_1)+\frac{2}{5}u(q_2)+\frac{1}{5}u(q_3)\right)+\frac{1}{5}\left(\frac{2}{5}u(q_1)+\frac{1}{5}u(q_2)+\frac{2}{5}u(q_3)\right)\\
&=\frac{16}{25}u(q_1)+\frac{5}{25}u(q_2)+\frac{4}{25}u(q_3)\ge\frac{7}{25}u(q_1),
\end{align*}
and
$$u(q_8)=\frac{16}{25}u(q_1)+\frac{4}{25}u(q_2)+\frac{5}{25}u(q_3)\ge\frac{7}{25}u(q_1).$$
By the maximum principle, we have
$$u\ge\frac{7}{25}u(q_1)>0\text{ on }W_0.$$
Hence
\begin{align*}
&\max\myset{|u(q_i)-u(q_j)|:i,j=1,2,3}\le 2u(q_1)\\
&\le2\cdot\frac{25}{7}\dashint_{W_0}u\md m\le\frac{50}{7}\frac{1}{3^{n-1}}\int_{2B}|u|\md m\le\frac{50\cdot3^{3}C_R}{7}\dashint_{2B}|u|\md m,
\end{align*}
hence
\begin{align*}
|\nabla u(x)|&=|u(p)-u(q)|\le\frac{50\cdot3^{3}C_R}{7}\left(\frac{3}{5}\right)^n\dashint_{2B}|u|\md m\\
&=\frac{250\cdot3^{2}C_R}{7}\frac{1}{(2^{n+1})^{\beta-\alpha}}\dashint_{2B}|u|\md m\le\frac{C}{r^{\beta-\alpha}}\dashint_{2B}|u|\md m,
\end{align*}
where $C=\frac{250\cdot3^{2}C_R}{7}$, hence
$$\mynorm{|\nabla u|}_{L^\infty(B)}\le\frac{C}{r^{\beta-\alpha}}\dashint_{2B}|u|\md m.$$
\end{proof}

\section{Proof of Theorem \ref{thm_grad}}\label{sec_grad}

We start by proving gradient estimates for the solutions of Poisson equation using

\noindent \ref{eqn_GRH}, see \cite[Theorem 3.2]{CJKS20} for a similar result using \ref{eqn_RH}.

\begin{myprop}\label{prop_Poi_grad}
Let $(X,d,m,\calE,\calF)$ be an unbounded cable system satisfying \ref{eqn_VPhi},

\noindent \hyperlink{eqn_LSPsiq}{LS($\Psi,q$)}, \ref{eqn_CS} and \ref{eqn_GRH}. Then for any $p\in\left[\frac{q}{q-1},+\infty\right)$, there exists $C\in(0,+\infty)$ such that for any ball $B=B(x_0,r)$, for any $f\in L^\infty(2B)$, if $u\in\calF$ satisfies $\Delta u=f$ in $2B$, then for $m$-a.e. $x\in B$, we have
$$|\nabla u(x)|\le C\left(\frac{\Phi(r)}{\Psi(r)}\dashint_{2B}|u|\md m+F_2(x)\right),$$
where
$$F_2(x)=\sum_{j\le[\log_2r]}\Phi(2^{j})\left(\dashint_{B(x,2^j)}|f|^p\md m\right)^{1/p}.$$
\end{myprop}

\begin{proof}
Let $x\in B\setminus V$ be fixed. Let $y\in B\setminus V$ be such that $[x,y]\subseteq B$, and $d(x,y)<\min\{\frac{r}{16},\frac{1}{16}\}$. If $y$ is close enough to $x$, then one can also find $p,q\in V$ such that $\{p,q\}\in E$ and $x,y\in(p,q)$. Let $k_0=[\log_2d(x,y)]$ and $k_1=[\log_2r]$, then $k_0+3\le k_1$. For any $k=k_0+3,\ldots,k_1$, by Lemma \ref{lem_Poi_exist}, there exists $u_k\in\calF_{B(x,2^k)}$ such that $\Delta u_k=f$ in $B(x,2^k)$ and
$$\dashint_{B(x,2^{k-1})}|u_k|\md m\lesssim\dashint_{B(x,2^k)}|u_k|\md m\lesssim\Psi(2^k)\left(\dashint_{B(x,2^k)}|f|^p\md m\right)^{1/p}.$$
Then
\begin{align*}
&|u(x)-u(y)|\\
&\le|(u-u_{k_1})(x)-(u-u_{k_1})(y)|\\
&+\sum_{k=k_0+4}^{k_1}|(u_k-u_{k-1})(x)-(u_k-u_{k-1})(y)|\\
&+|u_{k_0+3}(x)|+|u_{k_0+3}(y)|.
\end{align*}
For any $k=k_0+4,\ldots,k_1$, we have $d(x,y)<2^{k_0+1}\le2^{k-2}$, that is, $y\in B(x,2^{k-2})$.

\noindent Since $\Delta(u-u_{k_1})=0$ in $B(x,2^{k_1})$, by \ref{eqn_GRH}, we have
\begin{align*}
&|(u-u_{k_1})(x)-(u-u_{k_1})(y)|\\
&\le d(x,y)\mynorm{|\nabla(u-u_{k_1})|}_{L^\infty(B(x,2^{k_1-1}))}\\
&\lesssim d(x,y)\frac{\Phi(2^{k_1-1})}{\Psi(2^{k_1-1})}\dashint_{B(x,2^{k_1})}|u-u_{k_1}|\md m\\
&\le d(x,y)\frac{\Phi(2^{k_1-1})}{\Psi(2^{k_1-1})}\left(\dashint_{B(x,2^{k_1})}|u|\md m+\dashint_{B(x,2^{k_1})}|u_{k_1}|\md m\right)\\
&\lesssim d(x,y)\frac{\Phi(2^{k_1-1})}{\Psi(2^{k_1-1})}\left(\dashint_{2B}|u|\md m+\Psi(2^{k_1})\left(\dashint_{B(x,2^{k_1})}|f|^p\md m\right)^{1/p}\right)\\
&\lesssim d(x,y)\left(\frac{\Phi(r)}{\Psi(r)}\dashint_{2B}|u|\md m+\Phi(2^{k_1})\left(\dashint_{B(x,2^{k_1})}|f|^p\md m\right)^{1/p}\right).
\end{align*}

Similarly, for any $k=k_0+4,\ldots,k_1$, since $\Delta (u_k-u_{k-1})=0$ in $B(x,2^{k-1})$, by \ref{eqn_GRH}, we have
\begin{align*}
&|(u_k-u_{k-1})(x)-(u_k-u_{k-1})(y)|\\
&\le d(x,y)\mynorm{|\nabla(u_k-u_{k-1})|}_{L^\infty(B(x,2^{k-2}))}\\
&\lesssim d(x,y)\frac{\Phi(2^{k-2})}{\Psi(2^{k-2})}\dashint_{B(x,2^{k-1})}|u_k-u_{k-1}|\md m\\
&\le d(x,y)\frac{\Phi(2^{k-2})}{\Psi(2^{k-2})}\left(\dashint_{B(x,2^{k-1})}|u_k|\md m+\dashint_{B(x,2^{k-1})}|u_{k-1}|\md m\right)\\
&\lesssim d(x,y)\frac{\Phi(2^{k-2})}{\Psi(2^{k-2})}\left(\Psi(2^k)\left(\dashint_{B(x,2^k)}|f|^p\md m\right)^{1/p}+\Psi(2^{k-1})\left(\dashint_{B(x,2^{k-1})}|f|^p\md m\right)^{1/p}\right)\\
&\lesssim d(x,y)\left(\Phi(2^k)\left(\dashint_{B(x,2^k)}|f|^p\md m\right)^{1/p}+\Phi(2^{k-1})\left(\dashint_{B(x,2^{k-1})}|f|^p\md m\right)^{1/p}\right).
\end{align*}
Therefore,
\begin{align*}
&d(x,y)^{-1} \left(|(u-u_{k_1})(x)-(u-u_{k_1})(y)|\right.\\
&\left.\qquad\qquad+\sum_{k=k_0+4}^{k_1}|(u_k-u_{k-1})(x)-(u_k-u_{k-1})(y)|\right)\\
&\lesssim \frac{\Phi(r)}{\Psi(r)}\dashint_{2B}|u|\md m+\sum_{k=k_0+3}^{k_1}\Phi(2^k)\left(\dashint_{B(x,2^k)}|f|^p\md m\right)^{1/p},
\end{align*}
so letting $d(x,y)\downarrow0$, or equivalently, $k_0\to-\infty$, we have
\begin{align*}
&\varlimsup_{d(x,y)\downarrow0} d(x,y)^{-1} \left(|(u-u_{k_1})(x)-(u-u_{k_1})(y)|\right.\\
&\left.\qquad\qquad\qquad\qquad+\sum_{k=k_0+4}^{k_1}|(u_k-u_{k-1})(x)-(u_k-u_{k-1})(y)|\right)\\
&\lesssim \frac{\Phi(r)}{\Psi(r)}\dashint_{2B}|u|\md m+\sum_{k\leq k_1}\Phi(2^k)\left(\dashint_{B(x,2^k)}|f|^p\md m\right)^{1/p}.
\end{align*}
Since $\Delta u_{k_0+3}=f$ in $B(x,2^{k_0+3})$, by Lemma \ref{lem_Poi_point}, we have
\begin{align*}
&|u_{k_0+3}(x)|\\
&\lesssim\dashint_{B(x,2^{k_0+3})}|u_{k_0+3}|\md m+\sum_{j\le k_0+2}\Psi(2^j)\left(\dashint_{B(x,2^j)}|f|^p\md m\right)^{1/p}\\
&\lesssim\Psi(2^{k_0+3})\left(\dashint_{B(x,2^{k_0+3})}|f|^p\md m\right)^{1/p}+\sum_{j\le k_0+2}\Psi(2^j)\left(\dashint_{B(x,2^j)}|f|^p\md m\right)^{1/p}\\
&\le\sum_{j\le k_0+3}\Psi(2^j)\mynorm{f}_{L^\infty(2B)}\\
&=\sum_{j\le k_0+3}(2^j)^2\mynorm{f}_{L^\infty(2B)}\\
&\asymp2^{2k_0}\mynorm{f}_{L^\infty(2B)}\\
&\asymp d(x,y)^2\mynorm{f}_{L^\infty(2B)},
\end{align*}
where, in the fifth line, we use the fact that $d(x,y)<\frac{1}{16}$ which implies that $2^j\le1$ for any $j\le k_0+3$. 

Also, since $d(x,y)<2^{k_0+1}<2^{k_0+2}$, that is, $y\in B(x,2^{k_0+2})$, Lemma \ref{lem_Poi_point} implies by analogous computations that

\begin{align*}
|u_{k_0+3}(y)|&\lesssim\Psi(2^{k_0+3})\left(\dashint_{B(x,2^{k_0+3})}|f|^p\md m\right)^{1/p}+\sum_{j\le k_0+2}\Psi(2^j)\left(\dashint_{B(y,2^j)}|f|^p\md m\right)^{1/p}\\
&\le\sum_{j\le k_0+3}\Psi(2^j)\mynorm{f}_{L^\infty(2B)}\\
&\asymp d(x,y)^2\mynorm{f}_{L^\infty(2B)}.
\end{align*}
Consequently,
$$|u_{k_0+3}(x)|+|u_{k_0+3}(y)|\lesssim d(x,y)^2\mynorm{f}_{L^\infty(2B)}.$$
Letting $d(x,y)\downarrow0$, or equivalently, $k_0\to-\infty$, we have
$$\lim_{d(x,y)\downarrow0} d(x,y)^{-1}(|u_{k_0+3}(x)|+|u_{k_0+3}(y)|)=0.$$
Hence,
$$|\nabla u(x)|=\lim_{d(x,y)\downarrow0}\frac{|u(x)-u(y)|}{d(x,y)}\lesssim \frac{\Phi(r)}{\Psi(r)}\dashint_{2B}|u|\md m+\sum_{k\le k_1}\Phi(2^k)\left(\dashint_{B(x,2^k)}|f|^p\md m\right)^{1/p}.$$
\end{proof}

According to the main idea of the proof of \cite[Theorem 3.2]{Jiang15}, gradient estimates for the solutions of Poisson equation can be used to derive gradient estimate for the heat kernel. Thanks to the result of Proposition \ref{prop_Poi_grad}, we can now apply this idea to our setting, thus completing the proof of Theorem \ref{thm_grad}: 

\begin{proof}[Proof of Theorem \ref{thm_grad}]
By \cite[THEOREM 4]{Dav97}, we have the following estimate of the time derivative of the heat kernel:
\begin{equation}\label{eqn_hkt}
\left|\frac{\dd}{\dd t}p_t(x,y)\right|\le\frac{1}{tV\left(x,\Psi^{-1}(C_1t)\right)}\exp\left(-\Upsilon\left(C_2d(x,y),t\right)\right).
\end{equation}
In particular, $y\mapsto \frac{\dd}{\dd t}p_t(x,y)$ is an $L^\infty$ function, for any $t>0$ and $m$-a.e. $x\in X$.

For $m$-a.e. $x\in X$, the function $(t,y)\mapsto p_t(x,y)$ is a solution of the heat equation $\Delta_yp_t(x,y)+\frac{\dd}{\dd t}p_t(x,y)=0$ (here we use $\Delta_y,\nabla_y$ to mean that the operators act on the variable $y$). For any $r\in(0,+\infty)$, by Proposition \ref{prop_Poi_grad}, for $m$-a.e. $y\in X$, we have
\begin{align*}
&|\nabla_yp_t(x,y)|\\
&\lesssim\frac{\Phi(r)}{\Psi(r)}\dashint_{B(y,2r)}p_t(x,z)m(\md z)+\sum_{j\le[\log_2r]}\Phi(2^j)\left(\dashint_{B(y,2^j)}|\frac{\dd}{\dd t}p_t(x,z)|^pm(\md z)\right)^{1/p}.
\end{align*}
Letting $r=\Psi^{-1}(t)$, we have
\begin{align*}
&|\nabla_yp_t(x,y)|\\
&\lesssim\frac{\Phi(\Psi^{-1}(t))}{t}\dashint_{B(y,2\Psi^{-1}(t))}p_t(x,z)m(\md z)\\
&+\sum_{j\le[\log_2\Psi^{-1}(t)]}\Phi(2^j)\left(\dashint_{B(y,2^j)}|\frac{\dd}{\dd t}p_t(x,z)|^pm(\md z)\right)^{1/p}.
\end{align*}
We now distinguish two cases: first, we assume that $d(x,y)\ge4\Psi^{-1}(t)$; then, for any $z\in B(y,2\Psi^{-1}(t))$, we have $d(x,z)\ge\frac{1}{2}d(x,y)$, and for any $j\le[\log_2\Psi^{-1}(t)]$, for any $z\in B(y,2^j)$, we have $d(x,z)\ge\frac{1}{2}d(x,y)$. By \ref{eqn_UHK}, we therefore obtain
\begin{align*}
\dashint_{B(y,2\Psi^{-1}(t))}p_t(x,z)m(\md z)&\le\dashint_{B(y,2\Psi^{-1}(t))}\frac{1}{V(x,\Psi^{-1}(C_1t))}\exp\left(-\Upsilon\left(C_2d(x,z),t\right)\right)m(\md z)\\
&\le\dashint_{B(y,2\Psi^{-1}(t))}\frac{1}{V(x,\Psi^{-1}(C_1t))}\exp\left(-\Upsilon\left(\frac{C_2}{2}d(x,y),t\right)\right)m(\md z)\\
&=\frac{1}{V(x,\Psi^{-1}(C_1t))}\exp\left(-\Upsilon\left(\frac{C_2}{2}d(x,y),t\right)\right).
\end{align*}
Moreover, by Equation (\ref{eqn_hkt}), we have
\begin{align*}
&\sum_{j\le[\log_2\Psi^{-1}(t)]}\Phi(2^j)\left(\dashint_{B(y,2^j)}|\frac{\dd}{\dd t}p_t(x,z)|^pm(\md z)\right)^{1/p}\\
&\le\sum_{j\le[\log_2\Psi^{-1}(t)]}\Phi(2^j)\left(\dashint_{B(y,2^j)}\left(\frac{1}{tV\left(x,\Psi^{-1}(C_1t)\right)}\exp\left(-\Upsilon\left(C_2d(x,z),t\right)\right)\right)^pm(\md z)\right)^{1/p}\\
&\le\sum_{j\le[\log_2\Psi^{-1}(t)]}\Phi(2^j)\left(\dashint_{B(y,2^j)}\left(\frac{1}{tV\left(x,\Psi^{-1}(C_1t)\right)}\exp\left(-\Upsilon\left(\frac{C_2}{2}d(x,y),t\right)\right)\right)^pm(\md z)\right)^{1/p}\\
&=\sum_{j\le[\log_2\Psi^{-1}(t)]}\Phi(2^j)\frac{1}{tV\left(x,\Psi^{-1}(C_1t)\right)}\exp\left(-\Upsilon\left(\frac{C_2}{2}d(x,y),t\right)\right).
\end{align*}
For any $r\in(0,1)$, we have
$$\sum_{j\le[\log_2r]}\Phi(2^j)=\sum_{j\le[\log_2r]}2^j\asymp2^{[\log_2r]}\asymp r=\Phi(r),$$
while for any $r\in[1,+\infty)$, 
$$\sum_{j\le[\log_2r]}\Phi(2^j)=\sum_{j\le-1}2^j+\sum_{j=0}^{[\log_2r]}(2^j)^\alpha\asymp1+r^\alpha\asymp r^\alpha=\Phi(r).$$
Hence
$$\sum_{j\le[\log_2r]}\Phi(2^j)\asymp\Phi(r)\text{ for any }r\in(0,+\infty),$$
and consequently,
\begin{align*}
&\sum_{j\le[\log_2\Psi^{-1}(t)]}\Phi(2^j)\left(\dashint_{B(y,2^j)}|\frac{\dd}{\dd t}p_t(x,z)|^pm(\md z)\right)^{1/p}\\
&\lesssim\frac{\Phi(\Psi^{-1}(t))}{tV\left(x,\Psi^{-1}(C_1t)\right)}\exp\left(-\Upsilon\left(\frac{C_2}{2}d(x,y),t\right)\right).
\end{align*}
Finally, we obtain:
\begin{align*}
&|\nabla_yp_t(x,y)|\\
&\lesssim\frac{\Phi(\Psi^{-1}(t))}{t}\frac{1}{V(x,\Psi^{-1}(C_1t))}\exp\left(-\Upsilon\left(\frac{C_2}{2}d(x,y),t\right)\right)\\
&+\frac{\Phi(\Psi^{-1}(t))}{tV\left(x,\Psi^{-1}(C_1t)\right)}\exp\left(-\Upsilon\left(\frac{C_2}{2}d(x,y),t\right)\right)\\
&=\frac{2\Phi(\Psi^{-1}(t))}{tV\left(x,\Psi^{-1}(C_1t)\right)}\exp\left(-\Upsilon\left(\frac{C_2}{2}d(x,y),t\right)\right).
\end{align*}
The proof is thus complete in the case $d(x,y)\geq 4\Psi^{-1}(t)$.

We now assume that $d(x,y)<4\Psi^{-1}(t)$. Then,
\begin{align*}
&|\nabla_yp_t(x,y)|\\
&\lesssim\frac{\Phi(\Psi^{-1}(t))}{t}\dashint_{B(y,2\Psi^{-1}(t))}p_t(x,z)m(\md z)\\
&+\sum_{j\le[\log_2\Psi^{-1}(t)]}\Phi(2^j)\left(\dashint_{B(y,2^j)}|\frac{\dd}{\dd t}p_t(x,z)|^pm(\md z)\right)^{1/p}\\
&\le\frac{\Phi(\Psi^{-1}(t))}{t}\dashint_{B(y,2\Psi^{-1}(t))}\frac{1}{V(x,\Psi^{-1}(C_1t))}m(\md z)\\
&+\sum_{j\le[\log_2\Psi^{-1}(t)]}\Phi(2^j)\left(\dashint_{B(y,2^j)}\left(\frac{1}{tV(x,\Psi^{-1}(C_1t))}\right)^pm(\md z)\right)^{1/p}\\
&=\frac{\Phi(\Psi^{-1}(t))}{t}\frac{1}{V(x,\Psi^{-1}(C_1t))}+\sum_{j\le[\log_2\Psi^{-1}(t)]}\Phi(2^j)\frac{1}{tV(x,\Psi^{-1}(C_1t))}\\
&\asymp\frac{\Phi(\Psi^{-1}(t))}{t}\frac{1}{V(x,\Psi^{-1}(C_1t))}+\Phi(\Psi^{-1}(t))\frac{1}{tV(x,\Psi^{-1}(C_1t))}\\
&=\frac{2\Phi(\Psi^{-1}(t))}{tV(x,\Psi^{-1}(C_1t))}\\
&=\frac{2\Phi(\Psi^{-1}(t))}{tV(x,\Psi^{-1}(C_1t))}\exp\left(-\Upsilon\left(\frac{C_2}{2}d(x,y),t\right)\right)\exp\left(+\Upsilon\left(\frac{C_2}{2}d(x,y),t\right)\right)\\
&\le\frac{2\Phi(\Psi^{-1}(t))}{tV(x,\Psi^{-1}(C_1t))}\exp\left(-\Upsilon\left(\frac{C_2}{2}d(x,y),t\right)\right)\exp\left(+\Upsilon\left(2C_2\Psi^{-1}(t),t\right)\right),
\end{align*}
where
\begin{align*}
\Upsilon\left(2C_2\Psi^{-1}(t),t\right)=\sup_{s\in(0,+\infty)}\left(\frac{2C_2\Psi^{-1}(t)}{s}-\frac{t}{\Psi(s)}\right)=\sup_{s\in(0,+\infty)}\left(2C_2\frac{\Psi^{-1}(t)}{\Psi^{-1}(s)}-\frac{t}{s}\right)
\end{align*}
is bounded from above by some positive constant depending only on $C_2$ and $\beta$ (see Lemma \ref{lem_unibd} below), hence
$$|\nabla_yp_t(x,y)|\lesssim\frac{\Phi(\Psi^{-1}(t))}{tV(x,\Psi^{-1}(C_1t))}\exp\left(-\Upsilon\left(\frac{C_2}{2}d(x,y),t\right)\right).$$

Therefore
\begin{align*}
|\nabla_yp_t(x,y)|&\le\frac{\Phi(\Psi^{-1}(t))}{tV(x,\Psi^{-1}(C_1t))}\exp\left(-\Upsilon\left(C_2d(x,y),t\right)\right)\\
&\asymp
\begin{cases}
\frac{\Phi(\Psi^{-1}(t))}{tV(x,\Psi^{-1}(C_1t))}\exp\left(-C_2\frac{d(x,y)^2}{t}\right),&\text{if }t<d(x,y),\\
\frac{\Phi(\Psi^{-1}(t))}{tV(x,\Psi^{-1}(C_1t))}\exp\left(-C_2\left(\frac{d(x,y)}{t^{1/\beta}}\right)^{\frac{\beta}{\beta-1}}\right),&\text{if }t\ge d(x,y),\\
\end{cases}\\
&\le\begin{cases}
\frac{\Phi(\Psi^{-1}(t))}{tV(x,\Psi^{-1}(C_1t))}\exp\left(-C_2\frac{d(x,y)^2}{t}\right),&\text{if }t\in(0,1),\\
\frac{\Phi(\Psi^{-1}(t))}{tV(x,\Psi^{-1}(C_1t))}\exp\left(-C_2\left(\frac{d(x,y)}{t^{1/\beta}}\right)^{\frac{\beta}{\beta-1}}\right),&\text{if }t\in[1,+\infty),\\
\end{cases}\\
&\asymp
\begin{cases}
\frac{C_1}{\sqrt{t}V(x,\sqrt{t})}\exp\left(-C_2\frac{d(x,y)^2}{t}\right),&\text{if }t\in(0,1),\\
\frac{C_1}{t^{1-\frac{\alpha}{\beta}}V(x,t^{1/\beta})}\exp\left(-C_2\left(\frac{d(x,y)}{t^{1/\beta}}\right)^{\frac{\beta}{\beta-1}}\right),&\text{if }t\in[1,+\infty),\\
\end{cases}
\end{align*}
where in the third inequality, as in Subsection \ref{subsec_heat}, we use the facts that the function $\beta\mapsto\left(\frac{d(x,y)}{t^{1/\beta}}\right)^{\frac{\beta}{\beta-1}}$ is monotone decreasing if $d(x,y)>t$ and monotone increasing if $d(x,y)\le t$.
\end{proof}

The following lemma has been used in the above proof:

\begin{mylem}\label{lem_unibd}
Let $A\in(0,+\infty)$. Then there exists some positive constant $C$ depending only on $A$ and $\beta$ such that
$$\sup_{t,s\in(0,+\infty)}\left(A\frac{\Psi^{-1}(t)}{\Psi^{-1}(s)}-\frac{t}{s}\right)\leq C.$$
\end{mylem}

\begin{proof}
Denote
$$f(t,s)=A\frac{\Psi^{-1}(t)}{\Psi^{-1}(s)}-\frac{t}{s},\quad  t,s\in(0,+\infty).$$
By considering separately the cases $t,s\in(0,1)$, $t,s\in[1,+\infty)$, $0<t<1\le s$ and $0<s<1\le t$, one shows easily that 
$$f(t,s)\le A\max\left\{\left(\frac{t}{s}\right)^{1/2},\left(\frac{t}{s}\right)^{1/\beta}\right\}-\frac{t}{s}.$$
Since the function $(0,+\infty)\to\R$, $x\mapsto A\max\{x^{1/2}, x^{1/\beta}\}-x$ is bounded from above by some positive constant $C$ depending only on $A$ and $\beta$, one obtains
$$\sup_{t,s\in(0,+\infty)}f(t,s)\le C.$$
\end{proof}

As a consequence, we obtain an $L^p$-upper estimate for the gradient of the heat semi-group.

\begin{mycor}\label{cor_Gp}
For any $p\in(1,+\infty)$, there exists $C\in(0,+\infty)$ such that
$$\mynorm{|\nabla e^{-t\Delta}|}_{p\to p}\le
\begin{cases}
\frac{C}{\sqrt{t}},&\text{if }t\in(0,1),\\
\frac{C}{t^{1-\frac{\alpha}{\beta}}},&\text{if }t\in[1,+\infty).
\end{cases}
$$
\end{mycor}

\begin{proof}
We may assume that $t\in[1,+\infty)$ since the proof for $t\in(0,1)$ is similar. Taking $\gamma\in(0,+\infty)$, for any $f\in L^p(X;m)$, for $m$-a.e. $x\in X$, we have
\begin{align*}
&|\nabla e^{-t\Delta}f(x)|\le\int_X|\nabla_xp_t(x,y)|\cdot|f(y)|m(\md y)\\
&=\int_X|\nabla_xp_t(x,y)|\exp\left(\gamma\left(\frac{d(x,y)}{t^{1/\beta}}\right)^{\frac{\beta}{\beta-1}}\right)V(y,t^{1/\beta})^{1/p'}|f(y)|\\
&\cdot\exp\left(-\gamma\left(\frac{d(x,y)}{t^{1/\beta}}\right)^{\frac{\beta}{\beta-1}}\right)\frac{1}{V(y,t^{1/\beta})^{1/p'}}m(\md y)\\
&\le\left(\int_X|\nabla_xp_t(x,y)|^p\exp\left(\gamma p\left(\frac{d(x,y)}{t^{1/\beta}}\right)^{\frac{\beta}{\beta-1}}\right)V(y,t^{1/\beta})^{p/p'}|f(y)|^pm(\md y)\right)^{1/p}\\
&\cdot\left(\int_X\exp\left(-\gamma p'\left(\frac{d(x,y)}{t^{1/\beta}}\right)^{\frac{\beta}{\beta-1}}\right)\frac{1}{V(y,t^{1/\beta})}m(\md y)\right)^{1/p'}.
\end{align*}
By \ref{eqn_VD}, as in \cite[p. 944, line -5]{ACDH04} and \cite[Equation (2.5)]{CCFR17}, we have
$$\int_X\exp\left(-\gamma p'\left(\frac{d(x,y)}{t^{1/\beta}}\right)^{\frac{\beta}{\beta-1}}\right)\frac{1}{V(y,t^{1/\beta})}m(\md y)\lesssim1.$$
Hence,
\begin{align*}
&\int_X|\nabla e^{-t\Delta}f(x)|^pm(\md x)\\
&\lesssim\int_X\int_X|\nabla_xp_t(x,y)|^p\exp\left(\gamma p\left(\frac{d(x,y)}{t^{1/\beta}}\right)^{\frac{\beta}{\beta-1}}\right)V(y,t^{1/\beta})^{p/p'}|f(y)|^pm(\md y)m(\md x)\\
&=\int_X\left(\int_X|\nabla_xp_t(x,y)|^p\exp\left(\gamma p\left(\frac{d(x,y)}{t^{1/\beta}}\right)^{\frac{\beta}{\beta-1}}\right)m(\md x)\right)V(y,t^{1/\beta})^{p/p'}|f(y)|^pm(\md y).
\end{align*}
By Theorem \ref{thm_grad},
\begin{align*}
&\int_X|\nabla_xp_t(x,y)|^p\exp\left(\gamma p\left(\frac{d(x,y)}{t^{1/\beta}}\right)^{\frac{\beta}{\beta-1}}\right)m(\md x)\\
&\le\int_X\frac{C_1^p}{t^{(1-\frac{\alpha}{\beta})p}V(y,t^{1/\beta})^p}\exp\left(-pC_2\left(\frac{d(x,y)}{t^{1/\beta}}\right)^{\frac{\beta}{\beta-1}}\right)\exp\left(\gamma p\left(\frac{d(x,y)}{t^{1/\beta}}\right)^{\frac{\beta}{\beta-1}}\right)m(\md x).
\end{align*}
Take $\gamma\in(0,C_2)$, then \ref{eqn_VD} implies that
$$\int_X|\nabla_xp_t(x,y)|^p\exp\left(\gamma p\left(\frac{d(x,y)}{t^{1/\beta}}\right)^{\frac{\beta}{\beta-1}}\right)m(\md x)\lesssim\frac{1}{t^{(1-\frac{\alpha}{\beta})p}V(y,t^{1/\beta})^{p-1}},$$
hence
\begin{align*}
\int_X|\nabla e^{-t\Delta}f(x)|^pm(\md x)&\lesssim\int_X\frac{1}{t^{(1-\frac{\alpha}{\beta})p}V(y,t^{1/\beta})^{p-1}}V(y,t^{1/\beta})^{p/p'}|f(y)|^pm(\md y)\\
&=\frac{1}{t^{(1-\frac{\alpha}{\beta})p}}\int_X|f(y)|^pm(\md y),
\end{align*}
that is,
$$\mynorm{|\nabla e^{-t\Delta}f|}_{L^p(X;m)}\lesssim\frac{1}{t^{1-\frac{\alpha}{\beta}}}\mynorm{f}_{L^p(X;m)}.$$
\end{proof}

\section{{\large{Proof of the \texorpdfstring{$L^p$}{Lp}-Boundedness of Quasi-Riesz Transforms}}}\label{sec_Riesz}

This section is devoted to the proof of Theorem \ref{thm_Riesz}. First, we prove the $L^p$-boundedness of the local Riesz transform as follows. We need the following two results.

\begin{mylem}\label{lem_Riesz_local1}(\cite[Theorem 1.2]{CD99})
Let $(X,d,m,\calE,\calF)$ be an unbounded MMD space that admits a ``carr\'e du champ". Assume that \ref{eqn_VD} and the following local diagonal upper bound \ref{eqn_DUHKloc} of the heat kernel hold, that is, there exists $C\in(0,+\infty)$ such that
\begin{equation*}\label{eqn_DUHKloc}\tag*{DUHK(loc)}
p_t(x,x)\le\frac{C}{V(x,\sqrt{t})}
\end{equation*}
for $m$-a.e. $x\in X$, for any $t\in(0,1)$. Then the local Riesz transform $\nabla(I+\Delta)^{-1/2}$ is $L^p$-bounded for any $p\in(1,2]$.
\end{mylem}

\begin{mylem}\label{lem_Riesz_local2}(\cite[THEOREM 1.5]{ACDH04})
Let $(X,d,m,\calE,\calF)$ be an unbounded MMD space that admits a ``carr\'e du champ". Assume that \ref{eqn_VD} and the following local $L^2$-Poincar\'e inequality on balls \ref{eqn_PI2loc} hold, that is, for any $r_0\in(0,+\infty)$, there exists a positive constant $C(r_0)$ depending on $r_0$ such that for any ball $B=B(x,r)$ with $r\in(0,r_0)$, for any $u\in\calF$, we have
\begin{equation*}\label{eqn_PI2loc}\tag*{PI(2,loc)}
\int_B|u-u_B|^2\md m\le C(r_0)r^2\int_B|\nabla u|^2\md m.
\end{equation*}
If there exist $p_0\in(2,+\infty]$, $\delta\in[0,+\infty)$ and $C\in(0,+\infty)$ such that
$$\mynorm{|\nabla e^{-t\Delta}|}_{p_0\to p_0}\le\frac{Ce^{\delta t}}{\sqrt{t}}\text{ for any }t\in(0,+\infty),$$
then the local Riesz transform $\nabla(aI+\Delta)^{-1/2}$ is $L^p$-bounded for any $p\in(2,p_0)$ and $a\in(\delta,+\infty)$.
\end{mylem}

\begin{myrmk}
{\em 
Although the orginal version of the above two results was stated in the setting of Riemannian manifolds, the same proof easily adapts to our setting.
}
\end{myrmk}

\begin{proof}[Proof of the $L^p$-boundedness of $\nabla(I+\Delta)^{-1/2}$]
If $p\in(1,2]$, then by Equation (\ref{eqn_hk}), we have \ref{eqn_DUHKloc}. By Lemma \ref{lem_Riesz_local1}, we have $\nabla(I+\Delta)^{-1/2}$ is $L^p$-bounded.

If $p\in(2,+\infty)$, then since \hyperlink{eqn_HKPsi}{HK($\Psi$)} holds, by Proposition \ref{prop_HK}, we have \ref{eqn_PI} which implies \ref{eqn_PI2loc}. Take an arbitrary $p_0\in(p,+\infty)$. By Corollary \ref{cor_Gp}, we have
$$\mynorm{|\nabla e^{-t\Delta}|}_{p_0\to p_0}\le
\begin{cases}
\frac{C}{\sqrt{t}},&\text{if }t\in(0,1),\\
\frac{C}{t^{1-\frac{\alpha}{\beta}}},&\text{if }t\in[1,+\infty).
\end{cases}
$$
Since $\sup_{t\in[1,+\infty)}t^{\frac{\alpha}{\beta}-\frac{1}{2}}e^{-\frac{1}{2}t}\in(0,+\infty)$, for any $t\in[1,+\infty)$, we have
$$\frac{1}{t^{1-\frac{\alpha}{\beta}}}=\frac{1}{t^{1-\frac{\alpha}{\beta}}}\frac{1}{\frac{e^{\frac{1}{2}t}}{\sqrt{t}}}\frac{e^{\frac{1}{2}t}}{\sqrt{t}}=t^{\frac{\alpha}{\beta}-\frac{1}{2}}e^{-\frac{1}{2}t}\frac{e^{\frac{1}{2}t}}{\sqrt{t}}\le\left(\sup_{t\in[1,+\infty)}t^{\frac{\alpha}{\beta}-\frac{1}{2}}e^{-\frac{1}{2}t}\right)\frac{e^{\frac{1}{2}t}}{\sqrt{t}}.$$
Hence
$$\mynorm{|\nabla e^{-t\Delta}|}_{p_0\to p_0}\le C\max\myset{1,\sup_{t\in[1,+\infty)}t^{\frac{\alpha}{\beta}-\frac{1}{2}}e^{-\frac{1}{2}t}}\frac{e^{\frac{1}{2}t}}{\sqrt{t}}\text{ for any }t\in(0,+\infty).$$
By Lemma \ref{lem_Riesz_local2}, we have $\nabla(I+\Delta)^{-1/2}$ is $L^p$-bounded.
\end{proof}
We now prove the $L^p$-boundedness of the quasi-Riesz transform at infinity.

\begin{proof}[Proof of the $L^p$-boundedness of $\nabla e^{-\Delta}\Delta^{-\veps}$]
Note that
$$\nabla e^{-\Delta}\Delta^{-\veps}=\frac{1}{\Gamma(\veps)}\int_0^{+\infty}\nabla e^{-(1+t)\Delta}\frac{\md t}{t^{1-\veps}}.$$
For any $p\in(1,+\infty)$, for any $f\in L^p(X;m)$, by Corollary \ref{cor_Gp}, we have
\begin{align*}
\mynorm{|\nabla e^{-\Delta}\Delta^{-\veps}f|}_{L^p(X;m)}&\le\frac{1}{\Gamma(\veps)}\int_{0}^{+\infty}\mynorm{|\nabla e^{-(1+t)\Delta}f|}_{L^p(X;m)}\frac{\md t}{t^{1-\veps}}\\
&\lesssim\int_0^{+\infty}\frac{1}{(1+t)^{1-\frac{\alpha}{\beta}}}\frac{\md t}{t^{1-\veps}}\mynorm{f}_{L^p(X;m)}.
\end{align*}
Since $\veps\in(0,1-\frac{\alpha}{\beta})$, the above integral converges; this implies that $\nabla e^{-\Delta}\Delta^{-\veps}$ is $L^p$-bounded.
\end{proof}

\bibliographystyle{plain}

\begin{thebibliography}{10}

\bibitem{A}
G.~K. {Alexopoulos}.
\newblock {\em {Sub-Laplacians with drift on Lie groups of polynomial volume
  growth}}, volume 739.
\newblock Providence, RI: American Mathematical Society (AMS), 2002.

\bibitem{Ame21}
A.~{Amenta}.
\newblock {New Riemannian manifolds with \(L^p\)-unbounded Riesz transform for
  \(p > 2\)}.
\newblock {\em {Math. Z.}}, 297(1-2):99--112, 2021.

\bibitem{AB15}
S.~Andres and M.~T. Barlow.
\newblock Energy inequalities for cutoff functions and some applications.
\newblock {\em {J. Reine Angew. Math.}}, 2015(699):183--215, 2015.

\bibitem{AC05}
P.~{Auscher} and T.~{Coulhon}.
\newblock {Riesz transform on manifolds and Poincar\'e inequalities}.
\newblock {\em {Ann. Sc. Norm. Super. Pisa, Cl. Sci. (5)}}, 4(3):531--555,
  2005.

\bibitem{ACDH04}
P.~Auscher, T.~Coulhon, X.~T. Duong, and S.~Hofmann.
\newblock Riesz transform on manifolds and heat kernel regularity.
\newblock {\em Ann. Sci. \'Ecole Norm. Sup.}, 37(6):911--957, 2004.

\bibitem{BCLS}
D.~Bakry, T.~Coulhon, M.~Ledoux, and L.~Saloff-Coste.
\newblock Sobolev inequalities in disguise.
\newblock {\em Indiana Univ. Math. J.}, 44(4):1033--1074, 1995.

\bibitem{BCG01}
M.~Barlow, T.~Coulhon, and A.~Grigor’yan.
\newblock Manifolds and graphs with slow heat kernel decay.
\newblock {\em Invent. Math.}, 144(3):609--649, 2001.

\bibitem{B98}
M.~T. {Barlow}.
\newblock {Diffusions on fractals}.
\newblock In {\em {Lectures on probability theory and statistics. Ecole d'Et\'e
  de probabilit\'es de Saint-Flour XXV - 1995. Lectures given at the summer
  school in Saint-Flour, France, July 10-26, 1995}}, pages 1--121. Berlin:
  Springer, 1998.

\bibitem{BB89}
M.~T. {Barlow} and R.~F. {Bass}.
\newblock {The construction of Brownian motion on the Sierpinski carpet}.
\newblock {\em {Ann. Inst. Henri Poincar\'e, Probab. Stat.}}, 25(3):225--257,
  1989.

\bibitem{BB90}
M.~T. {Barlow} and R.~F. {Bass}.
\newblock {On the resistance of the Sierpi\'nski carpet}.
\newblock {\em {Proc. R. Soc. Lond., Ser. A}}, 431(1882):345--360, 1990.

\bibitem{BB92}
M.~T. {Barlow} and R.~F. {Bass}.
\newblock {Transition densities for Brownian motion on the Sierpinski carpet}.
\newblock {\em {Probab. Theory Relat. Fields}}, 91(3-4):307--330, 1992.

\bibitem{BBS90}
M.~T. {Barlow}, R.~F. {Bass}, and J.~D. {Sherwood}.
\newblock {Resistance and spectral dimension of Sierpinski carpets}.
\newblock {\em {J. Phys. A, Math. Gen.}}, 23(6):l253--l258, 1990.

\bibitem{BGK12}
M.~T. {Barlow}, A.~{Grigor'yan}, and T.~{Kumagai}.
\newblock {On the equivalence of parabolic Harnack inequalities and heat kernel
  estimates}.
\newblock {\em {J. Math. Soc. Japan}}, 64(4):1091--1146, 2012.

\bibitem{BP88}
M.~T. {Barlow} and E.~A. {Perkins}.
\newblock {Brownian motion on the Sierpinski gasket}.
\newblock {\em {Probab. Theory Relat. Fields}}, 79(4):543--623, 1988.

\bibitem{BF16}
F.~{Bernicot} and D.~{Frey}.
\newblock {Riesz transforms through reverse H\"older and Poincar\'e
  inequalities}.
\newblock {\em {Math. Z.}}, 284(3-4):791--826, 2016.

\bibitem{C96}
G.~{Carron}.
\newblock {In\'egalit\'es isop\'erim\'etriques de Faber-Krahn et
  cons\'equences}.
\newblock In {\em {Actes de la table ronde de g\'eom\'etrie diff\'erentielle en
  l'honneur de Marcel Berger, Luminy, France, 12--18 juillet, 1992}}, pages
  205--232. Paris: Soci\'et\'e Math\'ematique de France, 1996.

\bibitem{Chen15}
L.~{Chen}.
\newblock {Sub-Gaussian heat kernel estimates and quasi Riesz transforms for
  \(1\leq p\leq 2\)}.
\newblock {\em {Publ. Mat., Barc.}}, 59(2):313--338, 2015.

\bibitem{CCFR17}
L.~{Chen}, T.~{Coulhon}, J.~{Feneuil}, and E.~{Russ}.
\newblock {Riesz transform for \(1\leq p \leq 2\) without Gaussian heat kernel
  bound}.
\newblock {\em {J. Geom. Anal.}}, 27(2):1489--1514, 2017.

\bibitem{CD99}
T.~{Coulhon} and X.~T. {Duong}.
\newblock {Riesz transforms for \(1\leq p\leq 2\)}.
\newblock {\em {Trans. Am. Math. Soc.}}, 351(3):1151--1169, 1999.

\bibitem{CJKS20}
T.~{Coulhon}, R.~{Jiang}, P.~{Koskela}, and A.~{Sikora}.
\newblock {Gradient estimates for heat kernels and harmonic functions}.
\newblock {\em {J. Funct. Anal.}}, 278(8):67, 2020.
\newblock Id/No 108398.

\bibitem{Dav97}
E.~B. {Davies}.
\newblock {Non-Gaussian aspects of heat kernel behaviour}.
\newblock {\em {J. Lond. Math. Soc., II. Ser.}}, 55(1):105--125, 1997.

\bibitem{Dun04a}
N.~{Dungey}.
\newblock {Heat kernel estimates and Riesz transforms on some Riemannian
  covering manifolds}.
\newblock {\em {Math. Z.}}, 247(4):765--794, 2004.

\bibitem{Dun04b}
N.~{Dungey}.
\newblock {Some gradient estimates on covering manifolds}.
\newblock {\em {Bull. Pol. Acad. Sci., Math.}}, 52(4):437--443, 2004.

\bibitem{FOT11}
M.~{Fukushima}, Y.~{Oshima}, and M.~{Takeda}.
\newblock {\em {Dirichlet forms and symmetric Markov processes. 2nd revised and
  extended ed}}, volume~19.
\newblock Berlin: Walter de Gruyter, 2nd revised and extended ed. edition,
  2011.

\bibitem{Gri92}
A.~{Grigor'yan}.
\newblock {The heat equation on noncompact Riemannian manifolds}.
\newblock {\em {Math. USSR, Sb.}}, 72(1), 1992.

\bibitem{Gri94}
A.~{Grigor'yan}.
\newblock {Heat kernel upper bounds on a complete non-compact manifold}.
\newblock {\em {Rev. Mat. Iberoam.}}, 10(2):395--452, 1994.

\bibitem{Gri09}
A.~{Grigor'yan}.
\newblock {\em {Heat kernel and analysis on manifolds}}, volume~47.
\newblock Providence, RI: American Mathematical Society (AMS); Somerville, MA:
  International Press, 2009.

\bibitem{GH14a}
A.~{Grigor'yan} and J.~{Hu}.
\newblock {Heat kernels and Green functions on metric measure spaces}.
\newblock {\em {Can. J. Math.}}, 66(3):641--699, 2014.

\bibitem{GHL15}
A.~{Grigor'yan}, J.~{Hu}, and K.-S. {Lau}.
\newblock {Generalized capacity, Harnack inequality and heat kernels of
  Dirichlet forms on metric measure spaces}.
\newblock {\em {J. Math. Soc. Japan}}, 67(4):1485--1549, 2015.

\bibitem{GT12}
A.~{Grigor'yan} and A.~{Telcs}.
\newblock {Two-sided estimates of heat kernels on metric measure spaces}.
\newblock {\em {Ann. Probab.}}, 40(3):1212--1284, 2012.

\bibitem{HKKZ00}
Ben~M. {Hambly}, Takashi {Kumagai}, Shigeo {Kusuoka}, and Xian~Yin {Zhou}.
\newblock {Transition density estimates for diffusion processes on homogeneous
  random Sierpi\'nski carpets}.
\newblock {\em {J. Math. Soc. Japan}}, 52(2):373--408, 2000.

\bibitem{Jiang15}
R.~{Jiang}.
\newblock {The Li-Yau inequality and heat kernels on metric measure spaces}.
\newblock {\em {J. Math. Pures Appl. (9)}}, 104(1):29--57, 2015.

\bibitem{JKY14}
R.~{Jiang}, P.~{Koskela}, and D.~{Yang}.
\newblock {Isoperimetric inequality via Lipschitz regularity of
  Cheeger-harmonic functions}.
\newblock {\em {J. Math. Pures Appl. (9)}}, 101(5):583--598, 2014.

\bibitem{Kig89}
J.~{Kigami}.
\newblock {A harmonic calculus on the Sierpi\'nski spaces}.
\newblock {\em {Japan J. Appl. Math.}}, 6(2):259--290, 1989.

\bibitem{Kig01}
J.~{Kigami}.
\newblock {\em {Analysis on fractals. Paperback reprint of the hardback edition
  2001}}.
\newblock Cambridge: Cambridge University Press, paperback reprint of the
  hardback edition 2001 edition, 2008.

\bibitem{KZ92}
S.~{Kusuoka} and X.~Y. {Zhou}.
\newblock {Dirichlet forms on fractals: Poincar\'e constant and resistance}.
\newblock {\em {Probab. Theory Relat. Fields}}, 93(2):169--196, 1992.

\bibitem{LY86}
P.~{Li} and S.~T. {Yau}.
\newblock {On the parabolic kernel of the Schr\"odinger operator}.
\newblock {\em {Acta Math.}}, 156:154--201, 1986.

\bibitem{Sal90}
L.~{Saloff-Coste}.
\newblock {Analyse sur les groupes de Lie \`a croissance polyn\^omiale.
  (Analysis on Lie groups of polynomial growth)}.
\newblock {\em {Ark. Mat.}}, 28(2):315--331, 1990.

\bibitem{Sal92}
L.~{Saloff-Coste}.
\newblock {A note on Poincar\'e, Sobolev, and Harnack inequalities}.
\newblock {\em {Int. Math. Res. Not.}}, 1992(2):27--38, 1992.

\bibitem{Sal95}
L.~{Saloff-Coste}.
\newblock {Parabolic Harnack inequality for divergence form second order
  differential operators}.
\newblock {\em {Potential Anal.}}, 4(4):429--467, 1995.

\bibitem{S94}
Paolo~M. Soardi.
\newblock {\em Potential theory on infinite networks}, volume 1590 of {\em
  Lecture Notes in Mathematics}.
\newblock Springer-Verlag, Berlin, 1994.

\bibitem{Str83}
R.~S. {Strichartz}.
\newblock {Analysis of the Laplacian on a complete Riemannian manifold}.
\newblock {\em {J. Funct. Anal.}}, 52:48--79, 1983.

\bibitem{Str00}
R.~S. {Strichartz}.
\newblock {Taylor approximations on Sierpinski gasket type fractals}.
\newblock {\em {J. Funct. Anal.}}, 174(1):76--127, 2000.

\bibitem{THP14}
D.~{Tang}, R.~{Hu}, and C.~{Pan}.
\newblock {H\"older estimates of harmonic functions on a class of p.c.f.
  self-similar sets}.
\newblock {\em {Anal. Theory Appl.}}, 30(3):296--305, 2014.

\end{thebibliography}

\end{document}